\documentclass[11pt,a4paper]{scrartcl}

\PassOptionsToPackage{shortlabels}{enumitem}
\PassOptionsToPackage{mathfont=stix2}{mystyle-koma-enhanced}
\usepackage{mystyle-koma-enhanced}
\usepackage{tikz}
\usepackage{wrapfig}


\newcommand{\E}{\mathbb E}
\renewcommand{\P}{\mathbb P}

\newcommand{\1}{\mathbf 1}

\newcommand{\revision}[1]{#1}
\declaretheorem[style=mystyleplain,name=Theorem]{maintheorem}

\usetikzlibrary{arrows.meta,bending,decorations.markings}
\definecolor{coralc}{RGB}{216,90,48}
\definecolor{tealc}{RGB}{15,110,86}
\definecolor{purplec}{RGB}{83,74,183}
\definecolor{orangec}{RGB}{186,117,23}

\title{Infinite-mean Galton--Watson trees have zero random-interchange threshold}

\author{
Andreas Klippel\thanks{Johannes Gutenberg University Mainz, Germany.
  \href{mailto:anklippe@uni-mainz.de}{anklippe@uni-mainz.de}}
\and
Benjamin Lees \orcidlink{0000-0003-2657-5358}\thanks{School of Mathematics,
  University of Leeds, Leeds, LS2 9JT, UK.
  \href{mailto:b.t.lees@leeds.ac.uk}{b.t.lees@leeds.ac.uk}}
\and
Christian M\"onch \orcidlink{0000-0002-6531-6482}\thanks{Independent Researcher,
  64289 Darmstadt, Germany.
  \href{mailto:cmoench25@gmail.com}{cmoench25@gmail.com}}
}

\date{6 October 2026}

\begin{document}
\maketitle

\begin{abstract}
We prove that, on a Galton--Watson tree with almost surely finite offspring and
infinite offspring mean, the random interchange process has critical parameter
zero for infinite cycles, conditionally on survival. Combined with the finite-mean strict inequality
established in \href{https://arxiv.org/abs/2503.03319}{arXiv:2503.03319},
this gives a dichotomy for locally finite Galton--Watson trees with offspring mean in
\((1,\infty]\): the loop and link critical parameters coincide exactly in the
infinite-mean case. The proof reduces the tree problem to a finite permutation
estimate derived from peeling estimates for random surfaces.
\end{abstract}

\section{Introduction and main results}
\label{sec:intro}
\label{sec:model}
\label{sec:gwzero}

The random interchange process, introduced by Harris under the name
\emph{random stirring model} \cite{harris_markov_lattices_1972}, is a spatial
random permutation generated by Poisson transpositions along the edges of a
graph $G=(V,E)$. Its cycles can equivalently be traced as loops in \(V\times[0,1]\).

The geometry of these cycles has been studied on several families of graphs.
For the complete graph, Schramm related macroscopic cycles to the giant
component of the underlying random graph and identified their
Poisson--Dirichlet scaling limit \cite{schramm_compositions_2005}. Infinite
cycles on \(\mathbb Z^d\) were recently established in dimensions \(d\ge5\)
for sufficiently large \(\beta\) by Elboim and Sly
\cite{elboim2024infinitecyclesinterchangeprocess}. On regular trees, the phase
transition has been developed through work of Angel and Hammond, among others
\cite{angel_permutations,hammond_infinitecycles,betz_sharptrees}. In the Galton--Watson setting, \revision{Betz, Ehlert, and the second author} established sufficient conditions for a phase transition in a broader family of random loop models that includes the random interchange process \cite{BetzEhlLees18}.

The process is also connected to graphical representations of the quantum
Heisenberg ferromagnet \cite{powers_heisenberg,toth_pressure}. Related loop
representations arise for other quantum spin systems, including the
antiferromagnetic Heisenberg model and more general random loop frameworks
\cite{aizenman_nachtergaele,ueltschi_random_loops}. Our subject throughout is
the random interchange process itself.

Every infinite loop of the interchange process lies in an infinite component of the subgraph
formed by edges carrying at least one Poisson point. The loop critical
parameter is therefore at least the corresponding percolation critical
parameter. Their comparison depends on the geometry of the underlying graph:
Schramm's complete-graph result gives asymptotic equality, whereas strict
separation holds on infinite bounded-degree graphs with finite link critical parameter
\cite{muehlbacher_critical_2021,betzklippelkraft2026}.

In our recent paper \cite{klippel_lees_monch_loop_vs_bernoulli_trees}, we
proved strict inequality of the loop and link thresholds for Galton--Watson
trees with offspring mean in \((1,\infty)\), conditioned on survival. The
present paper treats the complementary infinite-mean case. Its link threshold
is zero, and our main result shows that the loop threshold is zero as well.
Thus, conditionally on survival, the loop and link thresholds on locally
finite Galton--Watson trees with offspring mean in \((1,\infty]\) are strictly
separated when the mean is finite, whereas in the infinite-mean case they
coincide and are both zero.

\paragraph{Model and critical parameters.}

Let \(T\) be a locally finite tree rooted at \(o\). In the random interchange process
with parameter \(\beta>0\), each edge \(e\) carries an independent Poisson
point process of intensity \(\beta\) on \([0,1]\). A point of this process is
called a cross or a link. We work on the almost-sure event that each vertex's copy of $[0,1]$, which we think of as a
circle, has finitely many incident cross endpoints and no two of them have the
same time.

Remove the cross endpoints from each circle \(\{x\}\times[0,1]\) and direct the
remaining cyclic intervals in increasing time. If an interval ends at a cross
on \(\{x,y\}\) at time \(t\), its successor is the interval on the \(y\)-circle
beginning immediately after \(t\); a circle with no incident endpoint gives
one interval whose successor is itself. This defines a bijection on the
countable set of directed intervals. A loop is an orbit of this bijection,
with the corresponding crosses inserted between successive intervals. For
finite graphs this is the usual trajectory construction; see, for instance,
\cite{ueltschi_random_loops}. On a locally finite infinite tree, this
formulation defines the orbits directly on the interval state space. The loop
from a point in the interior of an interval follows the orbit of that interval.
A cross endpoint is visited when the interval ending there belongs to the
orbit; a loop started at the endpoint begins with the interval reached
immediately after crossing it. On finite subconfigurations this agrees with the
usual trajectory rule.
\revision{Figure~\ref{fig:successor-bijection} illustrates the construction
for a graph consisting of one edge.}

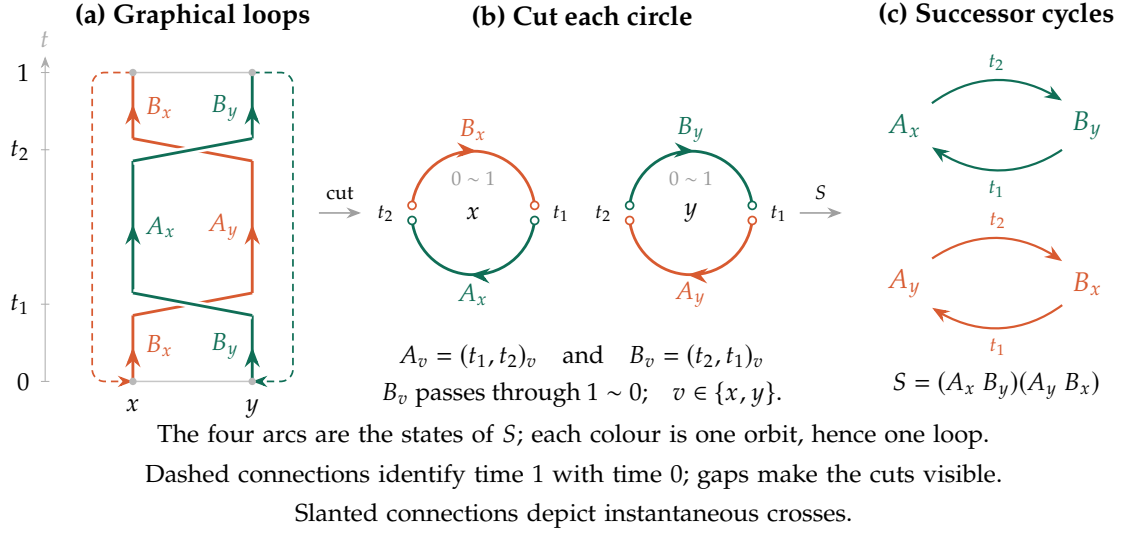
\begin{figure}[htbp]
\centering
\resizebox{\textwidth}{!}{%
\begin{tikzpicture}[
  x=1cm,y=1cm,>=Stealth,
  every node/.style={font=\small},
  strand/.style={line width=1.15pt},
  directed/.style={strand,postaction={decorate},
    decoration={markings,mark=at position 0.52 with {\arrow{Stealth}}}},
  seam/.style={line width=0.7pt,densely dashed,->},
  state/.style={inner sep=4pt,font=\normalsize},
  successor/.style={->,line width=0.9pt}
]
  % Panel headings.
  \node[font=\bfseries] at (1.65,5.6) {(a) Graphical loops};
  \node[font=\bfseries] at (7.2,5.6) {(b) Cut each circle};
  \node[font=\bfseries] at (13.1,5.6) {(c) Successor cycles};

  % (a) One edge, represented by two vertical time circles.
  % The slanted cross strands are a schematic separation of t^- and t^+.
  \draw[->,gray!65] (-0.45,0.4) -- (-0.45,5.03) node[above] {$t$};
  \foreach \yy/\lab in {0.4/0,1.5/{t_1},3.7/{t_2},4.8/1} {
    \draw[gray!65] (-0.51,\yy) -- (-0.39,\yy);
    \node[left] at (-0.54,\yy) {$\lab$};
  }
  \draw[gray!45,line width=0.6pt] (0.8,0.4) -- (2.5,0.4);
  \draw[gray!45,line width=0.6pt] (0.8,4.8) -- (2.5,4.8);

  % At x: B_x (coral), A_x (teal), B_x (coral).
  \draw[directed,coralc] (0.8,0.4) -- (0.8,1.34);
  \draw[directed,tealc] (0.8,1.66) -- (0.8,3.54);
  \draw[directed,coralc] (0.8,3.86) -- (0.8,4.8);
  % At y: B_y (teal), A_y (coral), B_y (teal).
  \draw[directed,tealc] (2.5,0.4) -- (2.5,1.34);
  \draw[directed,coralc] (2.5,1.66) -- (2.5,3.54);
  \draw[directed,tealc] (2.5,3.86) -- (2.5,4.8);

  % The two simultaneous switches at each cross. The white underlay makes
  % clear that the crossing strands do not form a junction.
  \draw[strand,coralc] (0.8,1.34) -- (2.5,1.66);
  \draw[strand,tealc,preaction={draw=white,line width=3.3pt}]
    (2.5,1.34) -- (0.8,1.66);
  \draw[strand,coralc] (2.5,3.54) -- (0.8,3.86);
  \draw[strand,tealc,preaction={draw=white,line width=3.3pt}]
    (0.8,3.54) -- (2.5,3.86);

  % Dashed wrap connections identify time 1 with time 0 at each vertex.
  \draw[seam,coralc,rounded corners=5pt]
    (0.8,4.8) -- (0.22,4.8) -- (0.22,0.4) -- (0.8,0.4);
  \draw[seam,tealc,rounded corners=5pt]
    (2.5,4.8) -- (3.08,4.8) -- (3.08,0.4) -- (2.5,0.4);
  \foreach \xx in {0.8,2.5} {
    \fill[gray!55] (\xx,0.4) circle[radius=1.5pt];
    \fill[gray!55] (\xx,4.8) circle[radius=1.5pt];
  }
  \node[below] at (0.8,0.28) {$x$};
  \node[below] at (2.5,0.28) {$y$};
  \foreach \yy in {0.91,4.3} {
    \node[right,coralc] at (0.85,\yy) {$B_x$};
    \node[left,tealc] at (2.45,\yy) {$B_y$};
  }
  \node[right,tealc] at (0.85,2.6) {$A_x$};
  \node[left,coralc] at (2.45,2.6) {$A_y$};

  \draw[->,gray!75,line width=0.65pt] (3.45,2.8) -- (4.05,2.8);
  \node[above,font=\scriptsize] at (3.75,2.86) {cut};

  % (b) The two open cyclic intervals on each circle.
  % Increasing time is clockwise; t_1 and t_2 are at the right and left.
  % A_v is the lower arc, B_v the upper arc through the time seam.
  \foreach \xx/\vertex/\acol/\bcol in {
    5.65/x/tealc/coralc,8.75/y/coralc/tealc} {
    \begin{scope}[shift={(\xx,2.8)}]
      \draw[directed,\acol] (-7:0.88)
        arc[start angle=-7,end angle=-173,radius=0.88];
      \draw[directed,\bcol] (173:0.88)
        arc[start angle=173,end angle=7,radius=0.88];
      \foreach \ang in {-7,-173}
        \draw[\acol,fill=white,line width=0.7pt]
          (\ang:0.88) circle[radius=1.6pt];
      \foreach \ang in {7,173}
        \draw[\bcol,fill=white,line width=0.7pt]
          (\ang:0.88) circle[radius=1.6pt];
      \node at (0,0) {$\vertex$};
      \node[\acol] at (0,-1.16) {$A_{\vertex}$};
      \node[\bcol] at (0,1.19) {$B_{\vertex}$};
      \node[font=\scriptsize,gray!80] at (0,0.49) {$0\sim1$};
      \node[right,font=\scriptsize] at (1.0,0) {$t_1$};
      \node[left,font=\scriptsize] at (-1.0,0) {$t_2$};
    \end{scope}
  }
  \node[align=center] at (7.2,0.48) {
    $A_v=(t_1,t_2)_v$\quad and\quad $B_v=(t_2,t_1)_v$\\[3pt]
    $B_v$ passes through $1\sim0$;\quad $v\in\{x,y\}$.
  };

  \draw[->,gray!75,line width=0.65pt] (10.3,2.8) -- (10.9,2.8);
  \node[above,font=\scriptsize] at (10.6,2.86) {$S$};

  % (c) Successor images verified from the terminal cross of each arc.
  % A_x --t_2--> B_y --t_1--> A_x and A_y --t_2--> B_x --t_1--> A_y.
  \node[state,tealc] (Ax) at (11.8,4.05) {$A_x$};
  \node[state,tealc] (By) at (14.4,4.05) {$B_y$};
  \draw[successor,tealc] (Ax.north east)
    to[bend left=35] node[above,font=\scriptsize] {$t_2$} (By.north west);
  \draw[successor,tealc] (By.south west)
    to[bend left=35] node[below,font=\scriptsize] {$t_1$} (Ax.south east);

  \node[state,coralc] (Ay) at (11.8,1.8) {$A_y$};
  \node[state,coralc] (Bx) at (14.4,1.8) {$B_x$};
  \draw[successor,coralc] (Ay.north east)
    to[bend left=35] node[above,font=\scriptsize] {$t_2$} (Bx.north west);
  \draw[successor,coralc] (Bx.south west)
    to[bend left=35] node[below,font=\scriptsize] {$t_1$} (Ay.south east);
  \node at (13.1,0.35) {$S=(A_x\ B_y)(A_y\ B_x)$};

  \node[align=center,text width=16.1cm] at (7.05,-0.95) {
    The four arcs are the states of $S$; each colour is one orbit, hence one loop.\\[4pt]
    Dashed connections identify time $1$ with time $0$; gaps make the cuts visible.\\[4pt]
    Slanted connections depict instantaneous crosses.
  };
\end{tikzpicture}
}
\caption{\revision{Two crosses on a single edge \(\{x,y\}\), at times
\(0<t_1<t_2<1\). Cutting the vertex circles produces four directed
intervals, whose two successor cycles are the two loops.}}
\label{fig:successor-bijection}
\end{figure}

Write \(L_x\) for the loop containing \((x,0)\). An infinite loop is a loop
whose orbit has infinite vertex support; write \(|L|\) for the cardinality of
the vertex support of a loop \(L\). The loop threshold is
\[
        \beta_c^{\mathrm{loop}}(T)
        =
        \inf\{\beta\ge0:\P_{T,\beta}(|L_o|=\infty)>0\}.
\]
The associated link percolation retains an edge if it carries at least one
cross, so the retention probability is \(1-e^{-\beta}\). Its threshold is
\[
        \beta_c^{\mathrm{link}}(T)
        =
        \inf\{\beta\ge0:\P_{T,\beta}(o\leftrightarrow\infty)>0\}.
\]
\revision{We use the convention \(\inf\varnothing=\infty\).}
Both thresholds are independent of the chosen root when \(T\) is connected.

The Galton--Watson tree in the main theorem has offspring variable \(Z\),
generating function \(f(s)=\E[s^Z]\), and unconditioned law
\(\mathbf P_Z\). We assume throughout that \(Z<\infty\) almost surely and
\(\E Z=\infty\), and condition the tree on survival. For a fixed realisation
\(T\), the probability \(\P_{T,\beta}\) averages only over the Poisson crosses.
We write \(\E_{T,\beta}\) for the corresponding expectation.
Thus \(\beta_c^{\mathrm{loop}}(T)\) and
\(\beta_c^{\mathrm{link}}(T)\) are quenched critical parameters, depending on
the realised tree. At fixed \(\beta\), write
\[
        \P_\beta^{\mathrm{ann}}(dT,d\omega)
        =
        \mathbf P_Z(dT)\P_{T,\beta}(d\omega)
\]
for the unconditioned annealed law and \(\E_\beta^{\mathrm{ann}}\) for its
expectation. We use this law for intermediate estimates. The conclusion of
Theorem~\ref{thm:gwzero} is quenched: for
\(\mathbf P_Z(\,\cdot\mid T\text{ survives})\)-almost every realisation
\(T\), both quenched critical parameters vanish. The final step of the proof
passes from positive annealed probability to this conditioned almost-sure
statement.

The proof uses clean two-cross children and a finite permutation estimate to
show that the multi-link loop visits linearly many of them in expectation. The
finite estimate is derived from a peeling estimate of
Budzinski, Curien, and Petri
\cite{budzinski_curien_petri_2019_universality}. Section~\ref{sec:proofs}
proceeds through the multi-link reduction, contraction to tagged endpoints,
Galton--Watson implication, and boundary-pairing argument.

\begin{maintheorem}[Absence of phase transition on infinite-mean Galton--Watson trees]
\label{thm:gwzero}
Let \(T\) be a Galton--Watson tree with offspring law \(Z\), conditioned on
survival, and assume that \(Z<\infty\) almost surely. If
\(\mathbb E Z=\infty\), then for the random interchange process
\[
        \beta_c^{\mathrm{loop}}(T)
        =
        \beta_c^{\mathrm{link}}(T)
        =
        0
\]
almost surely.
\revision{More precisely, there exists \(\beta_0>0\), depending only
on the offspring law, such that for every fixed
\(\beta\in(0,\beta_0)\),
\[
        \P_{T,\beta}(|L_o|=\infty)>0
\]
for
\(\mathbf P_Z(\,\cdot\mid T\text{ survives})\)-almost every realisation
\(T\).}
\end{maintheorem}

%\clearpage
\section{Proofs}
\label{sec:proofs}

\subsection{Multi-link exploration and finite input}
\label{subsec:multi}

Galton--Watson self-similarity was already used \revision{in work of the second author with Betz and Ehlert} to obtain recursive loop-survival estimates \cite[Section~3.1]{BetzEhlLees18}.
More specifically, our decomposition follows the multi-link-cluster exploration of \cite[Section~3]{betz_sharptrees}:
loops are explored inside components of edges carrying at least two links, while single-link boundary edges provide renewal points. We adapt this
\(d\)-ary-tree scheme to the unconditioned annealed Galton--Watson setting. For a vertex \(x\), let \(M_x\) be the connected component of \(x\)
\revision{in the full tree \(T\) after deleting every edge that carries fewer than two links. The \emph{multi-link loop} \(\gamma_x^s\)}
is the loop containing \((x,s)\) in the link configuration restricted to \(M_x\), and we write \(\gamma_x=\gamma_x^0\). A \emph{uni-link}
is a link on an edge that carries exactly one link. The support of a loop is the set of vertices whose circle intervals occur in its
directed orbit.

An \emph{outward one-link exit} of a multi-link loop is a child edge
\(\{v,w\}\), oriented away from the root, which carries exactly one cross at
time \(t\), with \((v,t)\) lying on the multi-link loop. If the loop has just
entered the present multi-link component through a uni-link from its parent,
that entering uni-link is not counted as an outward exit.

The following branching criterion is the one-sided annealed Galton--Watson
counterpart of the exploration criterion in \cite[Proposition~3.2 and Lemma~3.3]{betz_sharptrees}: the expected number of outward one-link exits controls loop survival.

\begin{lemma}[Annealed multi-link branching criterion]
\label{lem:multi}
\revision{Let \(T\) be an unconditioned Galton--Watson tree with offspring law \(Z\),
where \(Z<\infty\) almost surely, and run the random interchange process at a
fixed parameter \(\beta>0\). Let}
\(C^{(1)}(o)\) be the number of outward one-link exits of the multi-link loop
\(\gamma_o\). If
\[
        \E_\beta^{\mathrm{ann}}[C^{(1)}(o)]>1,
\]
then the full loop containing \((o,0)\) has infinite vertex support with
positive annealed probability:
\[
        \P_\beta^{\mathrm{ann}}(|L_o|=\infty)>0.
\]
\end{lemma}

\begin{proof}
If the multi-link loop \(\gamma_o\) has infinite vertex support with positive
\(\P_\beta^{\mathrm{ann}}\)-probability, Lemma~\ref{lem:addunilinks} implies that
the full loop is infinite with positive annealed probability. We may therefore
\revision{assume that \(\gamma_o\) is finite almost surely. Its finite vertex support,
local finiteness of the tree, and finiteness of the Poisson counts leave only
finitely many outward one-link exits. In particular, \(C^{(1)}(o)<\infty\)
almost surely.}

\revision{Start from \(\gamma_o\) and explore its outward one-link exits. Trace each
new multi-link loop and reveal the incident information needed to list its
exits, withholding the descendant environment beyond each exit until that
exit is processed. Contracting the multi-link components of the underlying
tree leaves a forest. Thus every outward uni-link enters a previously
unvisited component, and distinct exits lead to disjoint descendant
subtrees.}

\revision{When an exit enters a child \(x\), its parent edge carries only one cross
and is absent from the multi-link graph. Since \(T\) is a tree, the full-tree
component \(M_x\) is contained in the descendant subtree below \(x\).
The entering uni-link is excluded from the offspring count.
Let \(\mathcal E_n\) record the information revealed while processing
generations \(0,\ldots,n-1\), including the entry edges and times of the
unprocessed generation-\(n\) frontier. Conditional on this history, the
descendant environments at that frontier are independent copies of the
unconditioned Galton--Watson tree with its link processes: no offspring
variable or descendant link process beyond a frontier entrance has yet
been exposed. Each entry time is \(\mathcal E_n\)-measurable. Shifting by
that time preserves the conditional descendant Poisson law, so each fresh
exit count has the law of \(C^{(1)}(o)\), independently of the other counts.}

\revision{The generation sizes therefore form a Galton--Watson process with offspring
\(C^{(1)}(o)\). To include a possibly infinite offspring mean, put
\(g(s)=\E_\beta^{\mathrm{ann}}[s^{C^{(1)}(o)}]\).
Since \(g'(1-)>1\), there is \(s<1\) with \(g(s)<s\). Continuity gives a fixed
point below one, and the extinction probabilities obtained by iterating
\(g\) from zero increase to the smallest fixed point. This branching
exploration consequently survives with positive probability.}

\revision{Suppose now that the full root orbit were finite. Every outward uni-link
it crosses must be crossed back, because removing its edge disconnects the
tree. The unique cross on that edge makes the returning excursion resume
immediately after the same cross time. Contracting these returning excursions
in the finite full orbit gives the entire restricted multi-link loop in
the current component, so the full orbit crosses every outward exit of
that loop. Apply the same argument in each component reached through such
an exit. Inductively, all vertices of the branching exploration would lie
in the finite full-orbit support. On survival the exploration reaches
infinitely many distinct components, which is a contradiction. Thus, on
survival, the full root loop has infinite vertex support.}
\end{proof}

For our infinite-mean regime, the remaining input is the offspring estimate. The argument in
Subsection~\ref{subsec:gwimp} proves the hypothesis of Lemma~\ref{lem:multi} for arbitrarily small \(\beta>0\),
using boundary-pairing results. \revision{Since \(C^{(1)}(o)\) counts only one-link exits whose cross time is
visited by the multi-link loop, we estimate the time that \(\gamma_o\) spends
on children reached by exactly two crosses and with no multi-link child
edges continuing below them. We call them \emph{tagged clean double leaves}.}
We show that, when most
multi-link endpoints at a vertex come from such leaves, the multi-link loop
visits a positive fraction of them in expectation. \revision{The infinite
offspring mean then implies that the expected total time spent on tagged
children of the root is infinite. The exit calculation below gives
\(\E_\beta^{\mathrm{ann}}[C^{(1)}(o)]=\infty\), so the branching criterion applies.}

\revision{The following boundary-pairing estimate is the other key ingredient of the proof of Theorem~\ref{thm:gwzero}.} Its proof is postponed to
Subsection~\ref{subsec:finglu}.

\begin{proposition}[Boundary-pairing estimate]
\label{prop:small}
There are constants \(\delta_0>0\), \(a>0\), and \(p>0\) such that the
following holds. Let \(\chi\) be a deterministic permutation of a set \(H\)
with \(|H|=2K\). Suppose that the number of points of \(H\) lying in
\(\chi\)-cycles of length one or two is at most \(\delta_0 K\). Let \(\tau\)
be a uniform fixed-point-free involution on \(H\), independent of \(\chi\).
Then, uniformly in \(\chi\) and in every prescribed
\(j\in H\) which does not lie in a \(\chi\)-cycle of length one or two,
\[
        \mathbb P\bigl(|C_{\tau\chi}(j)|\ge aK\bigr)
        \ge p-o_K(1).
\]
\end{proposition}

\subsection{Contraction}
\label{subsec:endpt}

For this subsection fix \(\beta>0\) and a locally finite rooted tree \(T\), and
work under the quenched law \(\P_{T,\beta}\). For an independent uniform circle
time \(U\), write
\[
        \widehat{\P}_{T,\beta}
        =
        \P_{T,\beta}\otimes\operatorname{Leb}_{[0,1]},
        \qquad
        \widehat{\E}_{T,\beta}
        =
        \E_{\widehat{\P}_{T,\beta}}.
\]
On the annealed space used later, the same sigma-fields are enlarged by
\(\sigma(T)\).

At a vertex \(x\), a link endpoint is a point of \(\{x\}\times[0,1]\) incident
to a link. A multi-link endpoint at \(x\) is an endpoint of a link on an edge
incident to \(x\) which carries at least two links. Root-side endpoints are
endpoints coming from the edge to the parent of \(x\), while child endpoints
come from edges to children of \(x\). Let \(E_x^\downarrow\) be the edge set
of the descendant subtree rooted at \(x\); when \(x\ne o\), it does not contain the parent edge of \(x\). If
\(\Theta_e\) denotes the Poisson point configuration on \(e\), put
\[
        \mathcal R_x
        =
        \sigma\bigl(\Theta_e:e\notin E_x^\downarrow\bigr).
\]
This is the root-side sigma-field: it determines the endpoints and excursions
at \(x\) coming from the component toward the parent and is independent of the
link configurations on \(E_x^\downarrow\). For the root it is trivial. All
statements in this subsection hold conditionally on \(\mathcal R_x\).

Every child \(y\) of \(x\) such that the edge \(\{x,y\}\) carries exactly
two crosses and no edge from \(y\) to one of its children carries two or more
crosses is called a tagged clean double leaf. For each such child, the two
endpoints at \(x\) belonging to \(\{x,y\}\) are tagged endpoints. All remaining
multi-link endpoints at \(x\) are non-tagged. The non-tagged endpoints, including any conditioned root-side
endpoints, are used only through the first-return data they induce between
endpoints at \(x\). For a tagged clean double leaf \(y\) and a starting time
\(s\), write \(\ell_y(\gamma_x^s)\) for the Lebesgue measure of the set of
times on \(\{y\}\times[0,1]\) whose directed intervals belong to
\(\gamma_x^s\).

\begin{lemma}
\label{lem:finreturn}
Fix a vertex \(x\) of \(T\). Let
\(H_x\) be the set of multi-link endpoints at \(x\). Say that an endpoint
\(h\in H_x\) has finite first return if the directed multi-link orbit reached
immediately after crossing \(h\) contains a directed interval on the circle of
\(x\). The first such interval determines the return endpoint at \(x\). Let
\(D_x\) be the event that some endpoint in \(H_x\) has no finite first return.
Then
\[
\begin{split}
&\widehat{\P}_{T,\beta}\bigl(
        \gamma_x^U\text{ has infinite vertex support}
        \bigm|\mathcal R_x\bigr)>0 \hspace{5 pt}\text{almost surely on }
\bigl\{\P_{T,\beta}(D_x\mid\mathcal R_x)>0\bigr\}.
\end{split}
\]
\end{lemma}

\begin{proof}
The set \(H_x\) is finite almost surely, because \(T\) is locally finite and
each incident edge carries finitely many crosses. For \(h\in H_x\), let
\(I_h\) be the cyclic interval on the circle of \(x\) immediately preceding
\(h\) and containing no other endpoint of \(H_x\). The endpoint times are
distinct almost surely, so \(|I_h|>0\). If \(h\) has no finite first return and
\(U\in I_h\), then the directed multi-link orbit from
\((x,U)\) contains the interval \(I_h\), crosses \(h\), and then follows the
orbit reached immediately after \(h\). If the resulting loop had finite vertex
support, then local finiteness would give only finitely many directed intervals
in that support. The successor orbit would then be periodic. Since the
successor map is bijective and \(I_h\) is the predecessor of the state reached
after crossing \(h\), this periodic orbit would have to return to an interval
on the circle of \(x\), contradicting the choice of \(h\). Define
\[
        V_x
        =
        \max\{|I_h|:h\in H_x\text{ has no finite first return}\},
\]
with the maximum of the empty set equal to zero. Conditional on the link
configuration, the probability over \(U\) of an infinite multi-link loop is at
least \(V_x\), and \(V_x>0\) exactly on \(D_x\). Therefore
\[
        \widehat{\P}_{T,\beta}\bigl(
        \gamma_x^U\text{ has infinite vertex support}\bigm|\mathcal R_x\bigr)
        \ge
        \E_{T,\beta}[V_x\mid\mathcal R_x],
\]
and the right-hand side is strictly positive almost surely on
\(\{\P_{T,\beta}(D_x\mid\mathcal R_x)>0\}\).
\end{proof}

\begin{lemma}
\label{lem:addunilinks}
For \(\P_{T,\beta}\)-almost every link configuration and every \(s\in[0,1]\) not
equal to a link time incident to \(o\), if \(\gamma_o^s\) has infinite vertex
support, then the full loop through \((o,s)\) has infinite vertex support.
\end{lemma}

\begin{proof}
Compare the full directed orbit with \(\gamma_o^s\). Whenever a uni-link sends it
into the component not containing \(o\) after that edge is removed, an infinite child-side orbit proves the claim.
If the child-side orbit has finite vertex support, local finiteness makes its state space finite. The tree provides
no other edge back; without a return through the same uni-link, the entered state would have both an internal
predecessor in a finite directed cycle and the incoming predecessor, contradicting bijectivity. After the return,
the full orbit resumes \(\gamma_o^s\). By local finiteness, only finitely many uni-link endpoints lie between
successive multi-link endpoints of the restricted orbit. \revision{Hence either
some excursion has infinite support, or induction along those endpoints shows
that the full orbit reaches all of them after finite excursions; in both cases
its vertex support is infinite.}
\end{proof}

Now fix a vertex \(x\) and condition on \(\mathcal R_x\). At the root, Lemmas~\ref{lem:finreturn}
and~\ref{lem:addunilinks} handle \(\{\P_{T,\beta}(D_o)>0\}\). For the remaining analysis we work
on \(\{\P_{T,\beta}(D_x\mid\mathcal R_x)=0\}\), where every multi-link endpoint has finite first return almost surely.

\revision{Let \(T_x\) be the set of endpoints belonging to tagged clean double leaves
of \(x\), and put \(K_x=|T_x|/2\). Let \(B_x\) be the set of all other
multi-link endpoints at \(x\), and write \(H_x=T_x\sqcup B_x\).
Here \(H_x\) records the endpoint set; its points need not belong to one
loop. The following maps describe motion along the circle at \(x\) and
excursions away from \(x\) separately.
The permutation \(\sigma\) sends each endpoint in \(H_x\) to the next one
in increasing cyclic time on the circle of \(x\).
For \(a\in B_x\), leave \(x\) through the cross at \(a\) and follow the
multi-link trajectory until it first returns to \(x\). If it arrives at
endpoint \(b\), set \(\rho_{B_x}(a)=b\). Stop upon arrival, before any subsequent
motion along the circle at \(x\): thus \(\rho_{B_x}\) is a
\emph{departure-to-return map}. Since \(T\) is a tree, the return uses the
same edge, so \(\rho_{B_x}\) maps \(B_x\) into itself.
If two departures had the same returning endpoint, tracing the bijective
successor orbit backwards from that return to its preceding visit to \(x\)
would recover the same departure. Thus \(\rho_{B_x}\) is injective on the finite
set \(B_x\), and hence is a permutation. When \(H_x=\varnothing\), all
endpoint permutations are understood to be empty.
Figure~\ref{fig:endpoint-returns} illustrates these endpoint maps.}

\begin{figure}[tbp]
\centering
\begin{tikzpicture}[>=Stealth,x=0.72cm,y=0.72cm]
  \def\R{3.5}

  % --- boundary circle at x, and x itself ---
  \draw[gray!60] (0,0) circle (\R);
  \node[draw, circle, fill=gray!40, minimum size=2mm, inner sep=0pt] (X) at (0,0) {};
  \node[above] at (0,0.35) {$x$};

  % --- the 8 endpoints of H_x, in cyclic order ---
  % 1,2 = tagged pair (T_x); 3,7 = edge A endpoints (B_x)
  % 5,6 = tagged pair (T_x); 4,8 = edge B endpoints (B_x)
  \foreach \i/\ang/\col in {1/90/coralc,2/45/coralc,3/0/tealc,4/-45/tealc,
                            5/-90/coralc,6/-135/coralc,7/180/tealc,8/135/tealc} {
    \node[draw, circle, fill=\col, minimum size=5mm, inner sep=0pt] (P\i) at (\ang:\R) {};
    \node[font=\small] at (\ang:\R+0.7) {\i};
  }

  % --- sigma: cyclic order around the circle ---
  \foreach \i/\j in {1/2,2/3,3/4,4/5,5/6,6/7,7/8,8/1}
    \draw[->, gray!70, thick, shorten <=1pt, shorten >=2pt]
      (P\i) to[bend left=15] (P\j);

  % --- tau: inward bows related by a half-turn, clear of the circle arrows. ---
  % These are undirected tagged-pair relations; bar(rho_{B_x}) fixes their endpoints.
  \draw[dashed, coralc, thick, shorten <=2pt, shorten >=2pt]
    (P1) to[bend right=22] (P2);
  \draw[dashed, coralc, thick, shorten <=2pt, shorten >=2pt]
    (P5) to[bend right=22] (P6);

  % --- rho_{B_x}: identical inward curves, rotated by 45 degrees. ---
  % The alternating endpoint pairs require a crossing inside the disk.
  % A white underlay separates the curves at that crossing; it is not a junction.
  \draw[dashed, black, thick, <->, shorten <=2pt, shorten >=2pt]
    (P3) .. controls (2,-1.8) and (-2,-1.8) .. (P7);
  \begin{scope}[rotate=-45]
    \coordinate (returncontrolone) at (2,-1.8);
    \coordinate (returncontroltwo) at (-2,-1.8);
  \end{scope}
  \draw[dashed, black, thick, <->, shorten <=2pt, shorten >=2pt,
        preaction={draw=white,line width=3pt,solid,-}]
    (P4) .. controls (returncontrolone) and (returncontroltwo) .. (P8);

  % --- legend beside the circle, separating tagged pairs and non-tagged returns ---
  \begin{scope}[shift={(4.9,1.9)}, font=\small]
    \node[draw, circle, fill=coralc, minimum size=3mm, inner sep=0pt] at (0,0) {};
    \node[right] at (0.8,0) {\revision{$T_x=\{1,2,5,6\}$ -- tagged endpoints}};
    \node[draw, circle, fill=tealc, minimum size=3mm, inner sep=0pt] at (0,-0.7) {};
    \node[right] at (0.8,-0.7) {\revision{$B_x=\{3,4,7,8\}$ -- other endpoints}};
    \draw[dashed, coralc, thick] (-0.55,-1.5) -- (0.55,-1.5);
    \node[right] at (0.8,-1.5) {\revision{$\tau=(1\,2)(5\,6)$ -- tagged pairs}};
    \draw[->, gray!70, thick] (-0.55,-2.25) -- (0.55,-2.25);
    \node[right] at (0.8,-2.25) {\revision{$\sigma=(1\,2\,3\,4\,5\,6\,7\,8)$}};
    \draw[<->, black, dashed, thick] (-0.55,-3.175) -- (0.55,-3.175);
    \node[right] at (0.8,-3.175) {\revision{$\rho_{B_x}=(3\,7)(4\,8)$}};
    \node[right] at (0.8,-3.75) {\revision{departure-to-return map}};
  \end{scope}
\end{tikzpicture}
\caption{\revision{Endpoint maps on the circle at \(x\). The numbers label endpoints,
not tree vertices. The circle arrows give
\(\sigma=(1\,2\,3\,4\,5\,6\,7\,8)\); the orange dashed curves without
arrowheads give the tagged pairing \(\tau=(1\,2)(5\,6)\). The black dashed arrows illustrate
\(\rho_{B_x}=(3\,7)(4\,8)\), they encode off-\(x\) returns rather than tree edges.}}
\label{fig:endpoint-returns}
\end{figure}
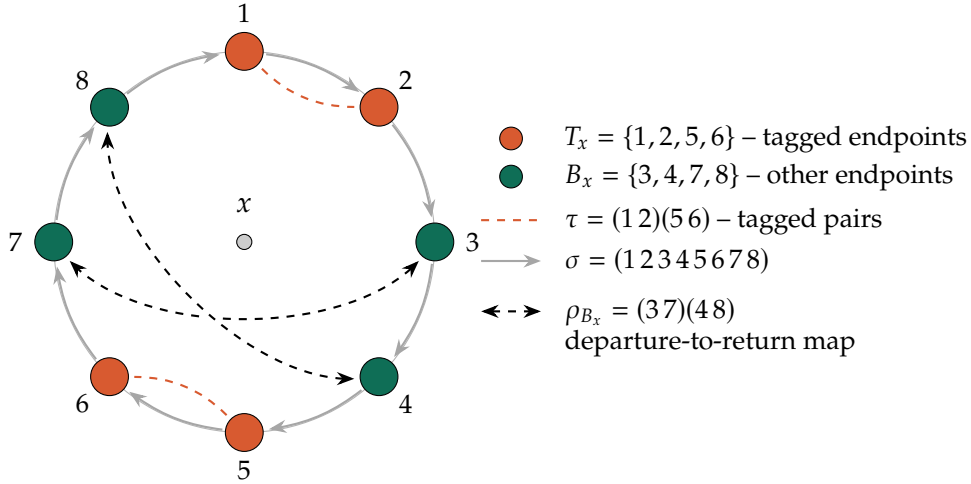

\revision{We now separate the motion along the circle and through non-tagged
excursions from the tagged excursions, whose pairing \(\tau\) will be
inserted later. Temporarily omit the tagged excursions while retaining
their endpoints as marked points on the circle at \(x\).
Starting just after an endpoint \(h\in H_x\), move along the circle to
\(a=\sigma(h)\). If \(a\in B_x\), complete the excursion departing at \(a\)
and stop at its return endpoint. If \(a\in T_x\), stop at that marked point
without taking its cross.

Extend \(\rho_{B_x}\) to \(\bar\rho_{B_x}\) on \(H_x\) by the identity on
\(T_x\). The permutation recording one such step is
\[
        \psi=\bar\rho_{B_x}\sigma,\qquad
        \psi(h)=
        \begin{cases}
        \sigma(h), & \sigma(h)\in T_x,\\
        \rho_{B_x}\bigl(\sigma(h)\bigr), & \sigma(h)\in B_x.
        \end{cases}
\]
Whether the step includes an excursion therefore depends on the endpoint
\(\sigma(h)\) just reached. For example, Figure~\ref{fig:endpoint-returns}
gives \(\psi(2)=\rho_{B_x}(3)=7\), whereas \(\psi(4)=5\).}
The first return of \(\psi\) to \(T_x\) is denoted by
\[
        \chi(t)=\psi^{m(t)}(t),
        \qquad
        m(t)=\min\{m\ge 1:\psi^m(t)\in T_x\}.
\]
\revision{Thus \(\chi\) follows the circle and non-tagged excursions from a
tagged return to the next tagged departure. Applying \(\tau\) then completes
that tagged excursion, so \(\tau\chi\) gives the next tagged return endpoint.}

We make the remaining conditioning explicit. For each edge \(e\), write \(N_e=|\Theta_e|\),
let \(\mathscr T_x\) be the set of edges from \(x\) to its tagged clean double leaves, and define the censored configuration
\[
        \Theta_e^{(x)}
        =
        \begin{cases}
        \dagger, & e\in\mathscr T_x,\\
        \Theta_e, & e\notin\mathscr T_x.
        \end{cases}
\]
Here \(\dagger\) is a fixed symbol outside the configuration space; when \(e\in\mathscr T_x\), the count \(N_e=2\) remains visible but both cross times are suppressed.
The set \(\mathscr T_x\) is determined by the cross counts. Set
\[
        \mathcal F_x
        =
        \sigma\bigl((N_e,\Theta_e^{(x)}):e\in E(T)\bigr).
\]
This field contains \(\mathcal R_x\) and the complete non-tagged first-return
data, while withholding precisely the two cross times on each edge in
\(\mathscr T_x\),
\revision{in particular, \(K_x\), \(B_x\), and \(\rho_{B_x}\) are \(\mathcal F_x\)-measurable.}

Conditional
on \(\mathcal F_x\), the \(2K_x\) still-hidden tagged cross times are independent
uniform points of \([0,1]\).
Let \(Q_x\) be their unordered set of positions on the circle of \(x\), and
put
\[
        \mathcal G_x
        =
        \mathcal F_x\vee\sigma(Q_x).
\]
The endpoint set \(H_x\), its cyclic order \(\sigma\), and the permutation
\(\chi\) are \(\mathcal G_x\)-measurable. The residual randomness assigns the
points of \(Q_x=T_x\) in pairs to the tagged children; denote the resulting
pairing by \(\tau\). The Dirichlet variables appearing in the proof of Lemma~\ref{lem:mass} below are the spacings between
the tagged points alone, before the non-tagged points are inserted. This choice
is important because the non-tagged return data can contain information about
the non-tagged endpoint times and thereby bias the full-circle spacings.

\begin{lemma}
\label{lem:endpt}
The map \(\chi:T_x\to T_x\) is a permutation. If \(b=|B_x|\), then
\[
        \#\operatorname{cycles}(\chi)\le b+1.
\]
Moreover, conditional on \(\mathcal G_x\), \(\tau\) is uniformly distributed over the
fixed-point-free involutions of \(T_x\). More precisely, for every bounded
function \(F\) of a matching,
\begin{equation}\label{eq:tau-uniform}
        \mathbb E[F(\tau)\mid\mathcal G_x]
        =
        \frac{1}{(2K_x-1)!!}
        \sum_{\vartheta\in\mathcal M(T_x)}F(\vartheta)
        \qquad\text{almost surely},
\end{equation}
where \(\mathcal M(T_x)\) is the set of fixed-point-free involutions of
\(T_x\), with the empty matching and \((-1)!!=1\) understood when \(K_x=0\).
\revision{The map \(\tau\chi\) takes each tagged return endpoint to the next
one along the loop.}
\end{lemma}

\begin{proof}
\revision{If \(H_x=\varnothing\), all assertions are immediate from
the empty-set conventions in the statement.  We henceforth assume that
\(H_x\ne\varnothing\).}
The map \(\psi\) is a permutation of the finite set \(H_x\). Hence the
\(\psi\)-orbit of any \(t\in T_x\) is a finite cycle containing \(t\), and so
returns to \(T_x\). The induced first-return map is therefore a permutation of
\(T_x\).

\revision{Since \(\sigma\) is one cycle and \(\bar\rho_{B_x}\) is supported on the \(b\)
non-tagged endpoints, \(\bar\rho_{B_x}\) can be written as a product of at most}
\(b\) transpositions. Multiplication by a transposition changes the number of
\revision{cycles of a permutation by at most one, so \(\psi=\bar\rho_{B_x}\sigma\) has at}
most \(b+1\) cycles. First return to \(T_x\) creates one cycle for each
\(\psi\)-cycle meeting \(T_x\), and therefore cannot increase this number.

Conditional on \(\mathcal F_x\), the hidden tagged cross times are independent
and uniform. After their unordered set \(Q_x\) has been revealed, every
assignment of its \(2K_x\) points to the two slots of each of the \(K_x\) labelled
tagged children has the same conditional density. Every
\(\vartheta\in\mathcal M(T_x)\) has the same number \(2^{K_x}K_x!\) of such labelled
assignments. Hence \eqref{eq:tau-uniform} holds. Since
\(\chi\) is \(\mathcal G_x\)-measurable, this also gives the required
conditional independence from \(\chi\).

\revision{Starting from a tagged return endpoint, \(\chi\) follows the circle
at \(x\) and all non-tagged first-return excursions until the next tagged
departure endpoint. The excursion through that tagged child then returns
at the endpoint specified by \(\tau\). Thus \(\tau\chi\) gives the next
tagged return endpoint.}
\end{proof}

\begin{lemma}
\label{lem:mass}
\revision{Assume that the \(\mathcal R_x\)-measurable event
\(\{\P_{T,\beta}(D_x\mid\mathcal R_x)=0\}\) holds, then all first-return maps
appearing below are defined almost surely.}
There are \(\eta_0>0\), \(c_0>0\), and a deterministic sequence
\((r_k)_{k\ge0}\) with \(0\le r_k\le c_0k\) and \(r_k/k\to0\) as
\(k\to\infty\) such that the
following holds. Let \(U\)
be uniform on \([0,1]\), independently of the link configuration. For every
\(0<\eta\le\eta_0\), define
\[
        G_x(\eta)=\{|B_x|\le \eta(2K_x+|B_x|)\}.
\]
Thus \(G_x(\eta)\) says that tagged endpoints form at least a \(1-\eta\)
fraction of all multi-link endpoints at \(x\). Almost surely on this \(\mathcal F_x\)-measurable event, one has
\[
        \widehat{\E}_{T,\beta}\Bigl[
        \sum_{y\ \mathrm{tagged}} \ell_y(\gamma_x^U)
        \,\Bigm|\, \mathcal F_x
        \Bigr]
        \ge c_0 K_x-r_{K_x} .
\]
\end{lemma}

\begin{proof}
\revision{We may take \(\eta_0<1\).  If \(K_x=0\), then
\(G_x(\eta)\) implies \(B_x=\varnothing\); on that event the sum in the
assertion is empty and the claimed bound is \(0\ge0\), since the construction
below gives \(r_0=0\).  Outside \(G_x(\eta)\) there is nothing to prove.
We may therefore assume that \(K_x\ge1\).}
Put \(b=|B_x|\), and let \(S_x\subset T_x\) be the set of tagged endpoints in
one- or two-cycles of \(\chi\). By Lemma~\ref{lem:endpt},
\[
        |S_x|\le 2(b+1).
\]
On \(G_x(\eta)\),
\[
        b\le \frac{2\eta}{1-\eta}K_x .
\]
Choosing \(\eta_0\) small enough and then taking \(K_x\) large makes
\(|S_x|\le \delta_0K_x\), where \(\delta_0\) is the constant from
Proposition~\ref{prop:small}. Let \(a,p>0\) be the other constants from that
proposition, decreasing \(p\) if necessary so that \(p\le1/2\), and choose a
deterministic sequence \((\zeta_k)_{k\ge1}\) decreasing to zero that dominates
the uniform \(o_k(1)\) error there. Increase the deterministic lower cutoff for
\revision{\(K_x\) so that \(\zeta_{K_x}\le p/8\).}

For \(d\in T_x\), let \(g_d\) be the length of the cyclic interval immediately
preceding \(d\) when only the tagged points are retained, and let
\(\widehat g_d\) be the length of the interval immediately preceding \(d\)
\revision{among all points of \(H_x\). Conditional on \(\mathcal F_x\), the unordered
collection of tagged-only gaps has the law of the coordinates of a
Dirichlet\((1,\ldots,1)\) vector: select a tagged point uniformly and rotate
it to zero; the other \(2K_x-1\) points remain independent uniforms, whose
ordered spacings have that law. Only quantities invariant under this
choice of origin will be used. For a nonnegative vector \(g\), write
\(R_m(g)\) for the sum of its \(m\) largest coordinates,}
with \(R_m(g)=R_{2K_x}(g)\) when \(m>2K_x\). At most \(b\) tagged-only intervals
contain a point of \(B_x\); hence
\[
        \sum_{d\in T_x}(g_d-\widehat g_d)
        \le R_b(g).
\]

Let \(\mathscr C_{\mathrm{mac}}\) be the cycles of the departure permutation
\(\chi\tau\) having size at least \(aK_x\), and set
\[
        W
        =
        \sum_{D\in\mathscr C_{\mathrm{mac}}}
        \sum_{d\in D}\widehat g_d .
\]
Since \((\chi\tau)^{-1}=\tau\chi^{-1}\), Proposition~\ref{prop:small},
applied to \(\chi^{-1}\), gives for every \(d\notin S_x\), almost surely,
\[
        \mathbb P\bigl(
        |C_{\chi\tau}(d)|\ge aK_x
        \,\bigm|\,\mathcal G_x\bigr)
        \ge p-\zeta_{K_x} .
\]
Consequently,
\begin{equation}\label{eq:W-preliminary}
\begin{aligned}
        \mathbb E[W\mid\mathcal G_x]
        &\ge
        (p-\zeta_{K_x})\sum_{d\notin S_x}g_d-R_b(g) \\
        &\ge
        (p-\zeta_{K_x})\bigl(1-R_{2(b+1)}(g)\bigr)-R_b(g).
\end{aligned}
\end{equation}

We next control the largest tagged-only gaps. Put \(n=2K_x\). If
\(g=(g_1,\ldots,g_n)\) has the Dirichlet\((1,\ldots,1)\) law and
\(m\le n\), then for \(1\le\lambda\le n/2\),
\[
        R_m(g)
        \le
        \frac{m\lambda}{n}
        +
        \sum_{i=1}^n g_i
        \mathbf 1_{\{g_i>\lambda/n\}} .
\]
The marginal tail is
\(\mathbb P(g_i>t)=(1-t)^{n-1}\). Integration of this tail gives
\[
\begin{aligned}
        \mathbb E R_m(g)
        &\le
        \frac{m\lambda}{n}
        +\lambda(1-\lambda/n)^{n-1}
        +(1-\lambda/n)^n  \\
        &\le
        \frac{m\lambda}{n}
        +(\lambda+1)e^{-\lambda/2}.
\end{aligned}
\]
It follows that there is a deterministic function
\(\omega(\alpha)\downarrow0\) as \(\alpha\downarrow0\) such that, uniformly
for large \(n\) and \(m/n\le\alpha\),
\[
        \mathbb E R_m(g)\le\omega(\alpha).
\]
Since
\[
        \frac b{2K_x}\le\frac{\eta}{1-\eta},
        \qquad
        \frac{2(b+1)}{2K_x}
        \le\frac{2\eta}{1-\eta}+\frac1{K_x},
\]
we may choose \(\eta_0\) small enough that, after taking the conditional
expectation in \eqref{eq:W-preliminary} over
\(\mathcal G_x\) given \(\mathcal F_x\),
\revision{the sum of the two expected error terms is at most
\(p/4\).  Since \(p-\zeta_{K_x}\ge7p/8\), this yields}
\begin{equation}\label{eq:W-mass}
        \mathbb E[W\mid\mathcal F_x]\ge\frac p2
\end{equation}
almost surely on \(G_x(\eta)\), for all sufficiently large \(K_x\).
\revision{The estimate \eqref{eq:W-mass} is uniform in the conditioning data
in \(\mathcal F_x\).}

\revision{We next use these departure cycles to bound the total time spent
on tagged leaves.} For
\(D\in\mathscr C_{\mathrm{mac}}\), let \(L(D)\) be the number of tagged clean
double leaves whose endpoint pair meets \(D\), and let
\[
        \widehat G(D)=\sum_{d\in D}\widehat g_d .
\]
\revision{The departure permutation \(\chi\tau\) is conjugate by \(\tau\) to
the return permutation \(\tau\chi\). Thus, if the uniform time \(U\) at \(x\)
lies in an interval counted by \(\widehat G(D)\), the loop's successive tagged
return endpoints follow the cycle corresponding to \(D\).} Since one leaf
contributes at most two endpoints to \(D\),
\[
        L(D)\ge\frac{|D|}{2}\ge\frac{aK_x}{2}.
\]

Fix \(0<\varepsilon<1/2\). For a tagged leaf \(y\), its two cross times on its
parent edge \(\{x,y\}\) delimit two complementary intervals of the time circle at \(y\).
Call \(y\) unbalanced if either interval has Lebesgue length less than \(\varepsilon\). Let \(\mathcal U\)
be the total number of such leaves, and let \(U(D)\) count those whose endpoint pair meets \(D\).
\revision{A balanced leaf met by \(D\) contributes at least \(\varepsilon\)
to the total time spent on tagged leaves.}
Hence, for every realisation of the link configuration and \(\tau\),
\[
\begin{aligned}
        \mathbb E_U
        \sum_{y\ \mathrm{tagged}}\ell_y(\gamma_x^U)
        &\ge
        \varepsilon
        \sum_{D\in\mathscr C_{\mathrm{mac}}}
        \widehat G(D)\bigl(L(D)-U(D)\bigr) \\
        &\ge
        \frac{\varepsilon aK_x}{2}W-2\varepsilon\mathcal U .
\end{aligned}
\]
The last inequality uses \(\widehat G(D)\le1\) and the fact that one
unbalanced pair meets at most two departure cycles.

Conditional on \(\mathcal F_x\), the two cross times belonging to each tagged
child are independent uniform points. Therefore
\[
        \mathbb E[\mathcal U\mid\mathcal F_x]
        =2\varepsilon K_x.
\]
Combining \(\mathbb E[\mathcal U\mid\mathcal F_x]=2\varepsilon K_x\) with
\eqref{eq:W-mass} gives, for all sufficiently large \(K_x\),
\[
        \widehat{\E}_{T,\beta}\Bigl[
        \sum_{y\ \mathrm{tagged}}\ell_y(\gamma_x^U)
        \,\Bigm|\,\mathcal F_x\Bigr]
        \ge
        \left(\frac{\varepsilon ap}{4}-4\varepsilon^2\right)K_x .
\]
Choose \(\varepsilon>0\) small enough that the coefficient is positive and
call a smaller positive constant \(c_0\). After increasing a deterministic
cutoff \(K_0\), define \(r_k=0\) for \(k\ge K_0\) and \(r_k=c_0k\) for
the finitely many \(k<K_0\). Then \(0\le r_k\le c_0k\) and
\(r_k/k\to0\); with this choice the asserted bound follows for every \(K_x\),
which proves the almost-sure conditional bound.
\end{proof}

\subsection{Galton--Watson implication}
\label{subsec:gwimp}

\revision{\begin{lemma}
\label{lem:aqbranch}
Let \(T\) be a locally finite rooted tree with root \(o\), and for a child
\(y\) of \(o\) let \(T_y\) be the rooted descendant subtree at \(y\).
For \(\beta>0\), set
\[
        B_\beta^{\mathrm{loop}}(T)
        =
        \{\P_{T,\beta}(|L_o|=\infty)>0\}.
\]
Then \(B_\beta^{\mathrm{loop}}(T)\) holds if and only if
\(B_\beta^{\mathrm{loop}}(T_y)\) holds for at least one child \(y\) of \(o\).
\end{lemma}}

\begin{proof}
For a fixed rooted tree \(S\), write
\[
        p_S(\beta)
        =
        \P_{S,\beta}(|L_o|=\infty).
\]
Thus \(p_S(\beta)\) is a quenched probability. We prove both implications at a
\revision{fixed parameter \(\beta>0\).}

First suppose that \(p_T(\beta)>0\). For a child \(y\), let \(I_y\) be the set
of times \(s\in[0,1]\) for which the directed loop through \((y,s)\) in the
isolated subtree \(T_y\) is infinite. By circular time-shift invariance,
\[
        p_{T_y}(\beta)
        =
        \E_{T_y,\beta}|I_y| .
\]
The root has finitely many incident edges and, at fixed \(\beta\), finitely
many incident crosses almost surely. Hence the root circle is split into
finitely many directed intervals. A directed orbit with infinite vertex support
cannot visit the same directed interval twice, because a repeated interval
would make the successor orbit periodic and therefore supported on finitely many
vertices. Thus an infinite loop can visit the root circle only finitely many
times. On the event that the root loop in \(T\) is infinite, let \(a\) be the
last root-circle interval in the forward orbit from \((o,0)\). The terminal
endpoint of \(a\) is a cross on some root-child edge \(\{o,y\}\), say at time
\(s\), and after crossing it the orbit never returns to the root.

\revision{Passing to the isolated subtree \(T_y\) deletes every cross
on the root-child edge, not only the cross at time \(s\).  The forward tail
after the crossing at \(s\) encounters no further endpoint of that edge on the
\(y\)-circle, since such an encounter would send the orbit back to the root.
Deleting all root-child crosses therefore leaves this forward tail unchanged;
it can only enlarge its first interval backwards across \(s\).  The isolated
loop through \((y,s)\) consequently contains the same infinite forward tail,
and hence \(s\in I_y\). Thus the root-child Poisson process has}
hit \(I_y\) for some child \(y\). Conditional on the child-subtree
configurations, the root-child crossing times are independent absolutely
continuous Poisson points. If \(|I_y|=0\) for every child, this hit event has
conditional probability zero. Since \(p_T(\beta)>0\) and the root has only
finitely many children, there is a child \(y\) with
\(\P_{T_y,\beta}(|I_y|>0)>0\), hence \(p_{T_y}(\beta)>0\). Therefore
\revision{\(B_\beta^{\mathrm{loop}}(T_y)\) holds.}

Conversely suppose that \(p_{T_y}(\beta)>0\) for some child \(y\). For this child,
let \(I_y\) be the set of times giving an infinite isolated \(T_y\)-loop. Then
\(\E_{T_y,\beta}|I_y|=p_{T_y}(\beta)>0\). On \(\{|I_y|>0\}\), require at the root that, among all
edges incident to \(o\), the only cross is on \(\{o,y\}\), and its time lies in
\(I_y\) but is not one of the finitely many child-subtree endpoint times on the
\revision{circle of \(y\). Its conditional probability is
\(\beta e^{-\beta\deg(o)}|I_y|\), which is positive whenever \(|I_y|>0\).
On this event, the loop from}
\((o,0)\) hits the unique cross to \(y\). Splitting an interval of an infinite
isolated \(T_y\)-loop at the crossing time cannot make the forward orbit from
the outgoing side finite; otherwise the original isolated directed orbit would
return to the same interval and be finite. Hence the full loop has infinite
vertex support. Thus \(p_T(\beta)>0\), so
\revision{\(B_\beta^{\mathrm{loop}}(T)\) holds.}
\end{proof}

\begin{proof}[Proof of Theorem~\ref{thm:gwzero}]
Fix \(0<\eta\le\eta_0\), where \(\eta_0\) is the constant from Lemma~\ref{lem:mass}.
\revision{Every subsequent restriction is an upper bound on
\(\beta\), so there is a deterministic \(\beta_0>0\) such that the argument
below applies to every \(0<\beta<\beta_0\). We first work}
under the unconditioned annealed law \(\P_\beta^{\mathrm{ann}}\).
For a fixed child \(y\) of a vertex \(x\), let \(P_y\) be the number of crosses
on \(\{x,y\}\), and set
\[
        q_\beta=(1+\beta)e^{-\beta}.
\]
The tagged-clean event has probability
\[
        a_\beta
        =
        e^{-\beta}\frac{\beta^2}{2} f(q_\beta).
\]
Write \(Z_x\) for the number of children of \(x\), and \(K_x\) for the number of its tagged clean double leaves.
Under \(\P_\beta^{\mathrm{ann}}\), \(K_x\mid Z_x\sim \operatorname{Bin}(Z_x,a_\beta)\).

\revision{Let \(t_\beta\) denote the expected number of tagged multi-link
endpoints at \(x\) contributed by one child. A child contributes two such
endpoints with probability \(a_\beta\), and zero otherwise. Thus
\[
        t_\beta:=2a_\beta=e^{-\beta}\beta^2 f(q_\beta),
\]}
while the expected total number of multi-link endpoints contributed by one
child is
\[
        \mu_\beta=\mathbb E[P\mathbf 1_{\{P\ge2\}}]
        =\beta(1-e^{-\beta}),
        \qquad P\sim \operatorname{Poisson}(\beta).
\]
Since \(f(q_\beta)\to1\) as \(\beta\downarrow0\), we have
\[
        t_\beta/\mu_\beta\to1 .
\]
\revision{Decrease \(\beta_0\), if necessary, so that
\(t_\beta/\mu_\beta>1-\eta/2\) for every
\(0<\beta<\beta_0\), and fix any such \(\beta\).}

At the root, let
\[
        K_o=\#\{\hbox{tagged clean double leaves of }o\},
        \qquad
        N_o=\sum_{y\text{ child of }o} P_y\mathbf 1_{\{P_y\ge2\}} .
\]
Thus \(N_o\) is the total number of child multi-link endpoints at \(o\), and
\(|B_o|=N_o-2K_o\). Conditional on \(Z_o=n\), the pairs
\[
        \bigl(2\mathbf 1_{\{y\text{ tagged}\}},
        P_y\mathbf 1_{\{P_y\ge2\}}\bigr)
\]
are i.i.d. with means \(t_\beta\) and \(\mu_\beta\). The law of large numbers
therefore gives
\[
        \frac{2K_o}{n}\to t_\beta,
        \qquad
        \frac{N_o}{n}\to\mu_\beta
\]
in \(\P_\beta^{\mathrm{ann}}(\,\cdot\mid Z_o=n)\)-probability as
\(n\to\infty\). Since \(t_\beta/\mu_\beta>1-\eta/2\), define
\(G_o(\eta)=\{N_o-2K_o\le\eta N_o\}\). This event has probability tending to one, and
\(K_o\ge a_\beta n/2\) holds with probability tending to one.

Let \(U\) be uniform on \([0,1]\), independently of the annealed tree-link
configuration, and write
\[
        \widehat{\P}_\beta^{\mathrm{ann}}
        =
        \P_\beta^{\mathrm{ann}}\otimes\operatorname{Leb}_{[0,1]},
        \qquad
        \widehat{\E}_\beta^{\mathrm{ann}}
        =
        \E_{\widehat{\P}_\beta^{\mathrm{ann}}}.
\]
Let \(D_\beta\) be the event that some multi-link endpoint at the root has no
finite first return. If
\(\P_\beta^{\mathrm{ann}}(D_\beta)>0\), disintegration over \(T\) gives a
set of positive \(\mathbf P_Z\)-measure on which
\(\P_{T,\beta}(D_\beta)>0\). On each such tree, Lemma~\ref{lem:finreturn} gives
positive probability that the multi-link loop from an independent uniform root time
has infinite support. Lemma~\ref{lem:addunilinks} then gives a full loop from the
same uniform root time with infinite support. Global circular time-shift invariance
of each quenched link law gives
\(\P_\beta^{\mathrm{ann}}(|L_o|=\infty)>0\).

Suppose instead that \(\P_\beta^{\mathrm{ann}}(D_\beta)=0\). Then all
first-return data at the root are finite almost surely, so
Lemma~\ref{lem:mass} applies on \(G_o(\eta)\), where tagged endpoints form at least
a \(1-\eta\) fraction of all root multi-link endpoints. \revision{On its
complement we use the lower bound zero.} The conditional bound in
Lemma~\ref{lem:mass} and the tower property therefore give
\begin{equation}\label{eq:root-mass-lower}
\begin{aligned}
&\widehat{\E}_\beta^{\mathrm{ann}}\Bigl[
        \sum_{y\ \mathrm{tagged}} \ell_y(\gamma_o^U)
        \,\Bigm|\, Z_o=n
        \Bigr] \\
&\qquad\ge
        \E_\beta^{\mathrm{ann}}\bigl[
        (c_0K_o-r_{K_o})\mathbf 1_{G_o(\eta)}
        \,\bigm|\, Z_o=n\bigr].
\end{aligned}
\end{equation}
The conditional law of large numbers gives
\(\P_\beta^{\mathrm{ann}}(G_o(\eta)\mid Z_o=n)\to1\) and
\(K_o\ge a_\beta n/2\) with conditional probability tending to one. Since
\(r_k=o(k)\), the right-hand side of \eqref{eq:root-mass-lower} is bounded below by
\(c_\beta n-o(n)\) for some \(c_\beta>0\). Thus
\[
        \widehat{\E}_\beta^{\mathrm{ann}}\Bigl[
        \sum_{y\ \mathrm{tagged}} \ell_y(\gamma_o^U)
        \,\Bigm|\, Z_o=n
        \Bigr]
        \ge c_\beta n-o(n)
\]
More explicitly, there are constants
\(c_\beta'>0\) and \(n_0<\infty\) such that the left-hand side of \eqref{eq:root-mass-lower} is
at least \(c_\beta' n\) for all \(n\ge n_0\), while it is always nonnegative.
The Poisson link configuration is invariant under a global circular time shift,
and the tagged-clean status and leaf-time sum are carried equivariantly by this
shift. Therefore the same conditional expectation is obtained with \(U\)
replaced by \(0\), i.e. with \(\gamma_o=\gamma_o^0\). Since
\(\mathbb E[Z\mathbf 1_{\{Z\ge n_0\}}]=\infty\),
\[
        \E_\beta^{\mathrm{ann}}
        \sum_{y\ \mathrm{tagged}}\ell_y(\gamma_o)=\infty .
\]

Conditional on \(y\) being a tagged clean double leaf, the expected number of
uni-link child edges below \(y\) is
\[
        b_\beta
        =
        \beta e^{-\beta}\frac{f'(q_\beta)}{f(q_\beta)}.
\]
\revision{To see this directly, let \(J_y\) be the number of one-link
child edges below \(y\).  Conditional on \(Z_y=m\), independence of the child
edge counts gives
\[
        \mathbb E\bigl[J_y; N_e\le1\text{ for every child edge }e
        \mid Z_y=m\bigr]
        =m\beta e^{-\beta}q_\beta^{m-1}.
\]
Averaging over \(Z_y\) gives the numerator
\(\beta e^{-\beta}f'(q_\beta)\), while the clean event has probability
\(f(q_\beta)\), proving the displayed formula.}
This is finite and positive. Indeed \(q_\beta<1\), so
\(\sup_{m\ge1} m q_\beta^{m-1}<\infty\) and hence \(f'(q_\beta)<\infty\);
also \(f(q_\beta)>0\), and \(f'(q_\beta)>0\) unless the offspring law is
concentrated at zero, which is incompatible with \(\mathbb E Z=\infty\). These
uni-link children and their cross times are independent of the multi-link loop,
conditional on the tagged-clean event. Given the set of times on
\(\{y\}\times[0,1]\) visited by \(\gamma_o\), a uni-link child of \(y\)
contributes an outward one-link exit exactly when its unique cross time falls
in this visited set. Its conditional probability of doing so is therefore
\revision{\(\ell_y(\gamma_o)\).
The restricted multi-link configuration deletes all
zero- and one-link child edges of \(y\), so neither their number nor their
uniform cross times enter the construction of \(\gamma_o\).  Since all
summands below are nonnegative, Tonelli's theorem applies even though the
resulting expectation is infinite.
Hence}
\[
        \E_\beta^{\mathrm{ann}}[C^{(1)}(o)]
        \ge
        b_\beta\,
        \E_\beta^{\mathrm{ann}}
        \sum_{y\ \mathrm{tagged}}\ell_y(\gamma_o)
        =
        \infty .
\]
By Lemma~\ref{lem:multi}, \(\P_\beta^{\mathrm{ann}}(|L_o|=\infty)>0\). Thus,
in either alternative for \(D_\beta\), the root loop is infinite with positive annealed probability.
\revision{Since \(\beta\in(0,\beta_0)\) was arbitrary, the preceding
argument gives
\begin{equation}\label{eq:annealed-loop-positive}
        \P_\beta^{\mathrm{ann}}(|L_o|=\infty)>0
        \qquad\text{for every }\beta\in(0,\beta_0).
\end{equation}}

For link percolation, Bernoulli retention at any \(\beta>0\) preserves
infinite mean for the retained offspring distribution. The retained
Galton--Watson tree therefore survives with positive annealed probability, so
\begin{equation}\label{eq:annealed-link-positive}
        \P_\beta^{\mathrm{ann}}(o\leftrightarrow\infty)>0
        \qquad\text{for every }\beta>0.
\end{equation}

\revision{It remains only to pass from positive annealed probability to
the conditioned quenched statement.  Fix a deterministic
\(\beta\in(0,\beta_0)\) and, under the unconditioned tree law
\(\mathbf P_Z\), put
\[
        r_\beta
        =\mathbf P_Z\bigl(B_\beta^{\mathrm{loop}}\bigr).
\]
By Lemma~\ref{lem:aqbranch} and the Galton--Watson branching
property, conditional on \(Z\) the child subtrees are independent copies of
the unconditioned tree, and therefore
\[
        r_\beta=1-f(1-r_\beta).
\]
On the other hand, \eqref{eq:annealed-loop-positive} and
annealed disintegration give
\[
        0
        <\P_\beta^{\mathrm{ann}}(|L_o|=\infty)
        =\int \P_{T,\beta}(|L_o|=\infty)\,\mathbf P_Z(dT),
\]
so \(r_\beta>0\).  Let \(\xi\) be the extinction
probability, that is, the smallest solution of \(s=f(s)\) in \([0,1]\).
Because \(\mathbb E Z=\infty\), the offspring law gives positive mass to some
value at least two; hence \(f\) is strictly convex on \((0,1)\).  Together
with \(f'(1-)=\infty>1\), this implies that the only solutions of
\(s=f(s)\) in \([0,1]\) are \(\xi\) and \(1\).  Thus the only solutions of
\(r=1-f(1-r)\) are \(0\) and \(1-\xi\), and consequently
\(r_\beta=1-\xi\).}

\revision{On extinction the almost surely locally finite
Galton--Watson tree is finite, so \(B_\beta^{\mathrm{loop}}\) is contained in
the survival event.  Since both events have probability \(1-\xi\),
\[
        \mathbf P_Z\bigl(
        B_\beta^{\mathrm{loop}}\mid T\text{ survives}\bigr)=1.
\]
This proves the stronger fixed-parameter assertion in the
theorem for every fixed \(\beta\in(0,\beta_0)\).  Now choose a deterministic
sequence \(\beta_n\downarrow0\) inside \((0,\beta_0)\) and intersect the
corresponding probability-one events.  On that intersection,
\[
        \beta_c^{\mathrm{loop}}(T)=0.
\]
Every infinite loop is contained in an infinite link
component, so
\(0\le\beta_c^{\mathrm{link}}(T)\le\beta_c^{\mathrm{loop}}(T)\).
Consequently
\[
        \beta_c^{\mathrm{loop}}(T)
        =\beta_c^{\mathrm{link}}(T)=0
\]
for
\(\mathbf P_Z(\,\cdot\mid T\text{ survives})\)-almost every \(T\), which}
proves Theorem~\ref{thm:gwzero}.
\end{proof}

\subsection{Proof of the finite boundary-pairing estimate}
\label{subsec:finglu}

We prove Proposition~\ref{prop:small} directly from the peeling
description of a uniform boundary pairing.  This avoids passing through
a permutation with independently sampled transpositions.

Although the argument is permutation-theoretic, we use the polygon-gluing
terminology behind the peeling exploration of Budzinski, Curien, and Petri
\cite[Sections~3.1--3.2]{budzinski_curien_petri_2019_universality}.
\revision{Represent each permutation cycle as an oriented boundary component, or
\emph{hole}, its elements as the still-unglued \emph{sides} in cyclic order,
and \(\tau\) as a pairing of those sides. Revealing a pair glues its two
sides, removes them from the unpaired boundary, and splices the remaining
boundary orders; in particular, \(1\)- and \(2\)-cycles are holes of perimeter
one and two.}

For a permutation \(\pi\) of \(H\), put
\[
\begin{aligned}
        N&:=|H|=2K,\\
        S(\pi)&:=\{x\in H:\text{$x$ belongs to a \(1\)- or \(2\)-cycle
        of \(\pi\)}\}.
\end{aligned}
\]
With \(\delta:=\delta_0/2\), the hypothesis of
Proposition~\ref{prop:small} is \(|S(\chi)|\leq\delta N\).
All error terms below are uniform over the deterministic permutation
\(\chi\) and over the choice of \(j\notin S(\chi)\).

We next define the exploration and all of its terminology.  After
\(i\) matching pairs have been exposed, let \(H_i\) be the set of
unpaired sides and let \(\chi_i\) be the permutation whose cycles list
those sides in their current boundary order.  Thus
\[
        H_0=H,\qquad \chi_0=\chi,\qquad |H_i|=N-2i.
\]
A cycle of \(\chi_i\) is a \emph{boundary cycle}, also called a
\emph{hole}, and its number of sides is its \emph{perimeter}.  A
\emph{boundary corner} is an occurrence between two consecutive sides
of a boundary cycle.  We write it as an ordered pair
\[
        (\ell,r),\qquad r=\chi_i(\ell),
\]
and call \(\ell\) and \(r\) respectively its \emph{left side} and
\emph{right side}.

Figure~\ref{fig:boundary-splice} depicts the generic different-component
case; the algebraic update below also covers degeneracies.

\begin{wrapfigure}[11]{r}{0.39\textwidth}
\centering
\begin{tikzpicture}[
  x=0.8cm,
  y=0.7cm,
  >={Stealth[length=3.5pt,width=3pt]},
  every node/.style={font=\footnotesize},
  boundary/.style={gray!70,line width=0.7pt},
  paired/.style={coralc,line width=1.8pt},
  pairing/.style={coralc,dashed,line width=0.8pt},
  orientation/.style={gray!70,->,line width=0.8pt},
  process/.style={->,line width=0.8pt}
]
  % Two boundary cycles before exposure.
  \coordinate (A0) at (0,0.75);
  \coordinate (A1) at (1.2,0.75);
  \coordinate (A2) at (1.2,1.95);
  \coordinate (A3) at (0,1.95);
  \draw[boundary] (A0)--(A1)--(A2)--(A3)--cycle;
  \draw[paired] (A1)--(A2);
  \node[above right=4pt,inner sep=0pt] at (1.2,1.35) {\(\ell\)};
  \draw[orientation] (0.85,1.35)
    arc[start angle=0,end angle=300,radius=0.25];
  \fill (A2) circle (1.7pt);
  \node[above right=2pt] at (A2) {\(\star\)};

  \coordinate (B0) at (0,-1.35);
  \coordinate (B1) at (1.2,-1.35);
  \coordinate (B2) at (1.2,-0.15);
  \coordinate (B3) at (0,-0.15);
  \draw[boundary] (B0)--(B1)--(B2)--(B3)--cycle;
  \draw[paired] (B1)--(B2);
  \node[below right=4pt,inner sep=0pt] at (1.2,-0.75) {\(z\)};
  \draw[orientation] (0.85,-0.75)
    arc[start angle=0,end angle=300,radius=0.25];

  \draw[pairing]
    (1.2,1.35) .. controls (2.0,1.05) and (2.0,-0.45) ..
    (1.2,-0.75);
  \draw[process] (2.7,0.3)--(3.55,0.3)
    node[midway,above] {glue};

  % The merged boundary cycle after deleting the paired sides.
  \coordinate (C0) at (5.0,-1.25);
  \coordinate (C1) at (6.15,-0.55);
  \coordinate (C2) at (6.15,0.85);
  \coordinate (C3) at (5.0,1.55);
  \coordinate (C4) at (3.85,0.85);
  \coordinate (C5) at (3.85,-0.55);
  \draw[boundary] (C0)--(C1)--(C2)--(C3)--(C4)--(C5)--cycle;
  \draw[orientation] (5.25,0.15)
    arc[start angle=0,end angle=300,radius=0.25];
\end{tikzpicture}
\captionsetup{font=footnotesize,skip=2pt}
\caption{\revision{The thick orange sides \(\ell\) and \(z=\tau(\ell)\) are glued and
  disappear from the boundary; the right-hand cycle consists of the
  remaining sides.}}
\label{fig:boundary-splice}
\end{wrapfigure}
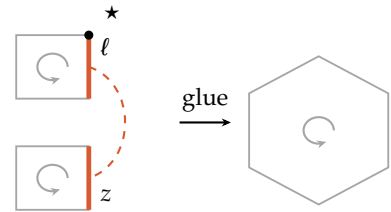

Algorithm~R of Budzinski, Curien, and Petri
\revision{\cite[Section~5]{budzinski_curien_petri_2019_universality}}
distinguishes one point of the boundary.  In the boundary-permutation
encoding we represent that point by a distinguished boundary corner,
called the \emph{red corner}; while following its boundary cycle we
also call it the \emph{active corner}.  In our initial state it is the
prescribed corner
\[
        (\ell_0,r_0):=(j,\chi(j)).
\]
It is useful to encode the red corner by a formal marker
\(\star\notin H_i\).  Whenever a red corner
\((\ell_i,r_i)\) is present, let \(\widehat\chi_i\) be the permutation
of \(H_i\cup\{\star\}\) obtained from \(\chi_i\) by replacing
\[
        \ell_i\longrightarrow r_i
        \quad\text{with}\quad
        \ell_i\longrightarrow\star\longrightarrow r_i.
\]
Thus \(\ell_i=\widehat\chi_i^{-1}(\star)\) and
\(r_i=\widehat\chi_i(\star)\).

At exposure \(i\), with \(i\) pairs already known, the algorithm
\emph{peels} the left side \(\ell_i\): it reveals its matching partner
\[
        z_i:=\tau(\ell_i).
\]
We also say that \(\ell_i\) is \emph{inspected} and that the pair
\(\{\ell_i,z_i\}\) is \emph{exposed}.
The other unpaired sides are the \emph{eligible partners}, and their
number is
\[
        D_i:=|H_i|-1=N-2i-1.
\]
Here is the boundary update, which we call the \emph{splice}.  If
\(\pi\) is a permutation of a finite set and \(A\) is a subset of that
set, write \(\operatorname{Del}_A(\pi)\) for the permutation obtained
by erasing the elements of \(A\) from the cycle notation of \(\pi\);
equivalently, the new successor of a retained point is the first
retained point encountered by iterating \(\pi\).  With permutations
composed from right to left, the exposed pair is spliced out by
\[
        H_{i+1}:=H_i\setminus\{\ell_i,z_i\},
        \qquad
        \chi_{i+1}
        :=\operatorname{Del}_{\{\ell_i,z_i\}}
          \bigl(\chi_i\circ(\ell_i\,z_i)\bigr).
 \tag{18}\label{eq:direct-splice}
\]
The red mark is updated by the same operation:
\[
        \widehat\chi_{i+1}^{\,\mathrm{pre}}
        :=\operatorname{Del}_{\{\ell_i,z_i\}}
          \bigl(\widehat\chi_i\circ(\ell_i\,z_i)\bigr).
\]
If the cycle of \(\star\) in
\(\widehat\chi_{i+1}^{\,\mathrm{pre}}\) contains a side of \(H_{i+1}\),
then \(\star\) marks a corner of the new boundary.  We take this corner
as the canonical continuation of the current red corner and define its
left and right sides by
\[
        \ell_{i+1}
          :=(\widehat\chi_{i+1}^{\,\mathrm{pre}})^{-1}(\star),
        \qquad
        r_{i+1}
          :=\widehat\chi_{i+1}^{\,\mathrm{pre}}(\star).
\]
If \(\star\) is instead a \(1\)-cycle, there is no canonical continuation
on the new boundary.  Deleting \(\star\) in either case gives
\(\chi_{i+1}\).  This convention covers every degenerate case.

For example, put
\[
        u_i:=\chi_i^{-1}(\ell_i),\qquad
        p_i:=\chi_i^{-1}(z_i),\qquad
        q_i:=\chi_i(z_i).
\]
If the six sides below are distinct, the two three-side segments in
boundary order are
\[
        u_i\longrightarrow\ell_i\longrightarrow r_i,
        \qquad
        p_i\longrightarrow z_i\longrightarrow q_i,
\]
then the splice removes \(\ell_i,z_i\) and replaces these segments by
\[
        u_i\longrightarrow q_i,
        \qquad
        p_i\longrightarrow r_i.
 \tag{19}\label{eq:direct-two-arc}
\]
A label repeats in \eqref{eq:direct-two-arc} only on a \(1\)- or
\(2\)-cycle or when the two segments overlap;
\eqref{eq:direct-splice} covers both cases.

The boundary cycle incident to the red corner is the
\emph{active boundary cycle}, and its perimeter before exposure \(i\)
is denoted by \(P_i\).  The splice and marker rule have the following
possible outcomes.
\begin{itemize}
\item If \(z_i\) lies on a different boundary cycle of perimeter
  \(s\), the two cycles merge into one of perimeter \(P_i+s-2\).
  The canonical continuation exists except when \(P_i=s=1\).
\item If \(z_i\) lies on the active boundary cycle, that cycle splits
  into two cycles whose perimeters sum to \(P_i-2\), with zero
  perimeters suppressed.  If the canonical continuation exists, its
  cycle is the new active boundary cycle.  For each possible new active
  perimeter, at most two choices of \(z_i\) produce that perimeter.
\end{itemize}
Pairing \(\ell_i\) to \(r_i\), the side immediately to the right of
the red corner, is a \emph{strong closure}; there is no canonical
continuation.
Pairing \(\ell_i\) to
\(u_i=\chi_i^{-1}(\ell_i)\ne r_i\), the adjacent side on its other
side, is a \emph{weak closure}; the canonical continuation exists.
When
\(P_i=1\) and \(z_i\) is the sole side of a
different boundary \(1\)-cycle, both cycles are removed from the
unpaired boundary, so there is no canonical continuation.  This is a
\emph{loop closure}.  Strong and loop closures are the only two cases
without a canonical continuation.

Whenever there is no canonical continuation and unpaired sides remain,
Algorithm~R resamples a boundary corner uniformly, independently of the
unexplored matching.  We prescribe the first corner as \((j,\chi(j))\);
choosing it uniformly and conditioning on its value would be equivalent.
We retain its canonical continuation until it ceases to exist and call
this segment the \emph{initial exploration}.  Its inspected sides belong
to \(C_{\tau\chi}(j)\), as Lemma~\ref{lem:direct-orbit} proves.
\revision{Resampling after every exposure would no longer ensure that all
inspected sides belong to the same product cycle.} It would still be \emph{Markovian}---its choices would use
only the exposed state and auxiliary randomness independent of the
unexplored matching---so the BCP loop/bigon estimates would remain valid.
After the initial exploration we continue only for estimates not tied to \(j\).

Let \(\mathcal F_i\) record the first \(i\) partner revelations and every
independent uniform replacement corner chosen after an exposure \(k<i\).
Thus \(H_i\), \(\chi_i\), and the red corner at time \(i\) are
\revision{\(\mathcal F_i\)-measurable, whereas \(z_i\) is not.  Proposition~11 of}
\cite{budzinski_curien_petri_2019_universality} says that, conditionally
on \(\mathcal F_i\), the partner \(z_i\) is uniform among the \(D_i\)
eligible sides.

Finally, let \(L_i\) and \(B_i\) be the respective numbers of boundary
\(1\)- and \(2\)-cycles after \(i\) exposures, and put
\[
        R_i:=L_i+2B_i.
\]
Thus \(R_i\) is the total \emph{short-boundary mass} at time \(i\): the
number of unpaired sides in all boundary \(1\)- and \(2\)-cycles.  We call
\(i\) an \emph{active short time} when the particular active cycle has
\(P_i\leq2\); its indicator is \(\1_{\{P_i\leq2\}}\).  All of
\(P_i,L_i,B_i,R_i\) are known before \(z_i\) is revealed.

\begin{lemma}
\label{lem:direct-orbit}
Fix \(\chi\), a perfect matching \(\tau\), and a starting side \(j\).
During the initial exploration, the left sides inspected
by Algorithm~R are distinct elements of the cycle
\(C_{\tau\chi}(j)\).  If the initial exploration ends by a strong
closure at exposure \(m\), then
\[
        |C_{\tau\chi}(j)|\geq m+1 .
\]
The only way in which this exploration can end without a strong closure
is a loop closure of the initial red corner.
\end{lemma}

\begin{proof}
Label the corner between \(s\) and \(\chi(s)\) by \(s\).
Orientation-reversing gluing of the sides \(s\) and \(\tau(s)\)
identifies the corner labelled \(s\) with the corner labelled
\(\chi^{-1}\tau(s)\).  Consequently two labelled corners become the
same vertex of the fully glued diagram exactly when their labels lie in
the same orbit of
\[
        \chi^{-1}\tau=(\tau\chi)^{-1}.
\]
A permutation and its inverse have the same orbits, differing only in
cyclic orientation.  Hence the corner labels identified at any one vertex
form exactly one orbit of \(\tau\chi\).

We now follow the active corner locally.  Suppose that the active
boundary contains the consecutive three-side segment
\[
        u\longrightarrow \ell\longrightarrow r
\]
at the red corner, and that the chosen partner \(z\) lies in the segment
\[
        p\longrightarrow z\longrightarrow q.
\]
After gluing \(\ell\) to \(z\) and deleting the paired sides, the two
segments are replaced by
\[
        u\longrightarrow q,
        \qquad
        p\longrightarrow r.
\]
In the non-overlapping case the active corner continues at the second
segment, with new left side
\[
        p=\chi_i^{-1}(z)=\chi_i^{-1}\tau(\ell),
\]
where \(\chi_i\) is the current boundary permutation.  This is the
next reverse-vertex step.  The splice deletes the exposed sides from
the current cyclic orders.  Hence, inductively, whenever the red corner
has a canonical continuation during the initial exploration, its new left
side \(\ell_{i+1}\) is the first still-unexposed element encountered after
\(\ell_i\) along the original \((\tau\chi)^{-1}\)-orbit.  If \(z=r\), then
\(\chi_i^{-1}\tau(\ell)=\ell\), so the move is precisely a strong
closure.

It remains to check the possible overlaps in this local description.
If \(z\) is the predecessor of \(\ell\), the adjacent weak splice gives
the same first-surviving reverse-orbit update.  If \(z\) is a boundary
\(1\)-cycle, the first reverse step lands at the deleted side \(z\);
the next surviving step continues the active corner unless the active
boundary is also a \(1\)-cycle, in which case the move is exactly a
loop closure of the initial red corner.  If \(z\) belongs to the
boundary \(2\)-cycle
\((z,w)\), the next surviving label is \(w\), the side inserted at the
active corner by the splice.  Finally, if the active boundary is a
\(2\)-cycle, pairing \(\ell\) to its other side is the strong closure,
and every external partner is covered by the preceding cases.  These
cases exhaust the degeneracies of the two-segment description.

It follows that, throughout the initial exploration, the
inspected sequence \(\ell_0=j,\ell_1,\ldots\) advances through the original
reverse orbit without repetition.  At a strong closure at exposure
\(m\), the \(m+1\) sides
\(\ell_0,\ldots,\ell_m\) are therefore distinct members of
\(C_{\tau\chi}(j)\).  The overlap analysis also proves the final
assertion.
\end{proof}

\begin{lemma}
\label{lem:direct-short}
Fix \(0<b<1/2\).  There are constants \(C_b,D_b<\infty\) such that, for
all sufficiently small fixed \(\delta>0\),
\[
 \mathbb P\left\{
   \max_{0\leq i\leq bN}R_i>D_b\delta N
 \right\}=o_N(1)
 \tag{20}\label{eq:direct-short-mass}
\]
and
\[
 \mathbb P\left\{
   \sum_{0\leq i<bN}\1_{\{P_i\leq2\}}
       >C_b\sqrt{\delta}\,N
 \right\}
 \leq C_b\sqrt{\delta}+o_N(1).
 \tag{21}\label{eq:direct-short-time}
\]
\end{lemma}

\begin{proof}
\revision{Proposition~13 of}
\cite{budzinski_curien_petri_2019_universality}, applied with
half-perimeter \(N/2\), gives the stochastic dominations
\[
 \sup_i L_i\preccurlyeq L_0+X_N,\qquad
 \sup_i B_i\preccurlyeq B_0+L_0+X_N,\qquad
 \sum_i\1_{\{P_i=2\}}\preccurlyeq B_0+L_0+X_N,
 \tag{22}\label{eq:direct-bcp-control}
\]
where \(\preccurlyeq\) denotes stochastic domination and
\(X_N=O_{\mathbb P}(\log N)\), uniformly in the initial configuration.
\revision{If simultaneously
\(\sup_iL_i\le L_0+t\) and
\(\sup_iB_i\le B_0+L_0+t\), then
\[
 \sup_iR_i
 \le (L_0+t)+2(B_0+L_0+t)
 =3L_0+2B_0+3t
 \le3\delta N+3t .
\]
The two separate stochastic dominations in
\eqref{eq:direct-bcp-control} and a union bound therefore give
\[
 \mathbb P\left\{
   \max_{0\le i\le bN}R_i>3\delta N+3t
 \right\}
 \le 2\,\mathbb P\{X_N>t\}.
\]
Taking \(t=\delta N/3\) and using
\(X_N=O_{\mathbb P}(\log N)\) proves
\eqref{eq:direct-short-mass}, for instance with \(D_b=4\).}

It remains to control the times at which the active boundary is a
\(1\)-cycle.  Equation~(11) of
\cite{budzinski_curien_petri_2019_universality}, translated from
half-perimeter \(N/2\), is
\[
 \mathbb E\left[
   \1_{\{P_i=1\}}+L_{i+1}-L_i
   \,\middle|\,\mathcal F_i
 \right]
 \leq \frac{2}{D_i}
      +\frac{2B_i\1_{\{P_i=1\}}}{D_i}.
 \tag{23}\label{eq:direct-one-cycle}
\]
\revision{Let \(m=\lceil bN\rceil\) and}
\(Y_1=\sum_{i<m}\1_{\{P_i=1\}}\).
Summing \eqref{eq:direct-one-cycle}, using \(D_i\geq(1-2b)N-1\),
and telescoping the \(L_i\)'s gives
\[
 \mathbb E Y_1
 \leq L_0+C_b+
 C_b\,\mathbb E\left[
       Y_1\,\frac{\max_{i<m}B_i}{N}
 \right].
 \tag{24}\label{eq:direct-absorb}
\]
Let
\[
        E_N:=\left\{\max_{0\leq i\leq bN}R_i\leq D_b\delta N\right\}.
\]
On \(E_N\), the last term is at most
\(C_bD_b\delta\,\mathbb E[Y_1;E_N]\).  On \(E_N^c\), the
\revision{deterministic bounds \(Y_1\leq\lceil bN\rceil\) and \(B_i\leq N/2\), together with}
\eqref{eq:direct-short-mass}, contribute \(o_N(N)\).  Choosing
\(\delta\) so that
\(C_bD_b\delta<1/2\), we may absorb the restricted term in
\eqref{eq:direct-absorb}; hence
\[
        \mathbb E Y_1\leq C_b\delta N+o_N(N).
\]
Markov's inequality yields
\[
        \mathbb P\{Y_1>C_b\sqrt{\delta}\,N\}
        \leq C_b\sqrt{\delta}+o_N(1).
\]
Combining this with the last estimate in
\eqref{eq:direct-bcp-control} proves
\eqref{eq:direct-short-time}.
\end{proof}

\begin{lemma}
\label{lem:direct-swallow}
\revision{Fix \(0<c<1/2\). There are constants \(C_c<\infty\) and}
\(\delta_c>0\) such that the following holds for
\revision{\(0<\delta\leq\delta_c\). Let \(\tau_{\mathrm{str}}\) and
\(\tau_{\mathrm{loop}}\) be the indices of the exposures causing the
respective strong and loop closures of the initial corner, with a missing
type set to infinity. A closure caused by \(z_i\) thus has index \(i\).
Then}
\[
 \mathbb P\{\tau_{\mathrm{loop}}<\tau_{\mathrm{str}},\
            \tau_{\mathrm{loop}}\leq cN\}
 \leq C_c\delta+o_N(1).
 \tag{25}\label{eq:direct-swallow}
\]
\end{lemma}

\begin{proof}
\revision{Choose \(D=D_c\) as in \eqref{eq:direct-short-mass}, and put
\[
 \tau_D:=\inf\{i:R_i>D\delta N\},\qquad
 \nu:=\tau_D\wedge(\tau_{\mathrm{str}}+1)\wedge(\tau_{\mathrm{loop}}+1).
\]
The event \(\{i<\nu\}\) means that the short-boundary mass has not exceeded
its threshold by time \(i\), and no closure occurred at an exposure before \(i\).
It is \(\mathcal F_i\)-measurable, so \(\nu\) is a stopping time for this
filtration. The added ones retain the possibility of a closure at the
unrevealed exposure \(i\).
Set
\[
 m:=\lfloor cN\rfloor+1,\qquad
 V(p):=r^p,\qquad W_i:=V(P_i)\1_{\{i<\nu\}},
\]
where \(r\in(0,1)\) will be fixed below, and consider \(i<m\).}

\revision{Condition on \(\mathcal F_i\), on the event \(\{i<\nu\}\), and put \(p=P_i\).
For a partner choice \(z\), let \(W_{i+1}(z)\) be the resulting value,
set to zero if that exposure causes either closure or makes
\(R_{i+1}>D\delta N\). Uniformity gives
\(\mathbb E[W_{i+1}\mid\mathcal F_i]=D_i^{-1}\sum_z W_{i+1}(z)\).
If \(z\) lies on a different non-short cycle of perimeter \(s\geq3\),
then \(P_{i+1}=p+s-2\geq p+1\), so this class contributes at most
\(rV(p)\) to this conditional expectation. At most \(D\delta N\) choices lie on different
short cycles, each with \(W_{i+1}(z)\leq r^{-1}V(p)\). Among non-stopping
choices on the active boundary, at most two give each resulting perimeter
\(q\geq1\). These latter two classes contribute at most, respectively,
\[
 \frac{D\delta N}{D_i}\,r^{-1}V(p),\qquad
 \frac{2}{D_i}\sum_{q\geq1}r^q\leq\frac{C_c}{N}.
\]
Since \(D_i\geq(1-2c)N-1\), first fix \(r<1\), and then take \(\delta>0\)
sufficiently small. The three bounds give constants \(\rho<1\) and
\(C_c<\infty\) such that
\[
 \mathbb E[W_{i+1}\mid\mathcal F_i]
 \leq \rho W_i+\frac{C_c}{N}\1_{\{i<\nu\}}.
 \tag{26}\label{eq:direct-drift}
\]
On \(\{i\geq\nu\}\), both \(W_i\) and \(W_{i+1}\) vanish, so the inequality
holds there as well. Taking expectations and iterating gives
\[
 \mathbb E W_i
 \leq \rho^i\mathbb E W_0+\frac{C_c}{N}\sum_{k=0}^{i-1}\rho^k
 \leq \rho^i+\frac{C_c}{N(1-\rho)}.
\]
Consequently,
\[
 \sum_{i<m}\mathbb E W_i
 \leq \frac{1}{1-\rho}+\frac{C_cm}{N(1-\rho)}
 \leq C_c.
\]
Since \(\1_{\{P_i=1,\ i<\nu\}}\leq r^{-1}W_i\), this proves
\[
 \mathbb E\sum_{i<m}\1_{\{P_i=1,\ i<\nu\}}\leq C_c.
 \tag{27}\label{eq:direct-one-occupation}
\]}

\revision{At any such pre-exposure state, a loop closure requires a partner on
another \(1\)-cycle. Its conditional probability is at most
\(D\delta N/D_i\leq C_c\delta\). Summing these adapted hazards and using
\eqref{eq:direct-one-occupation} yields
\[
 \mathbb P\{\tau_{\mathrm{loop}}<\tau_{\mathrm{str}},\
             \tau_{\mathrm{loop}}\leq cN,\
             \tau_{\mathrm{loop}}<\tau_D\}\leq C_c\delta.
\]
The sum includes every integer exposure at most \(cN\). Adding
\(\mathbb P\{\tau_D\leq cN\}=o_N(1)\) from
\eqref{eq:direct-short-mass} proves \eqref{eq:direct-swallow}.}
\end{proof}

\begin{lemma}
\label{lem:direct-clock}
Fix
\[
        0<a_1<a_2<b<1/2
\]
and set
\[
        q:=\sqrt{1-2a_1}-\sqrt{1-2a_2}>0.
\]
For all sufficiently small fixed \(\delta>0\),
\[
 \mathbb P\left\{
   |C_{\tau\chi}(j)|\geq \lceil a_1N\rceil
 \right\}
 \geq q-C_b\sqrt{\delta}-C_b\delta-o_N(1).
 \tag{28}\label{eq:direct-clock-bound}
\]
\end{lemma}

\begin{proof}
\revision{Enlarge \(\mathcal F_i\) by the auxiliary Bernoulli variables, if
any, from exposures strictly before \(i\); this preserves the uniform
unexplored matching.}
At exposure \(i\), before the partner is revealed, put
\[
        h_i:=\frac1{D_i}.
\]
If \(P_i\geq2\), the right side of the active corner is a specified
eligible partner.  We may therefore reveal the uniform partner in two
stages: \revision{first sample a Bernoulli variable \(\xi_i\)} with parameter
\(h_i\); on \(\{\xi_i=1\}\), choose that specified side, and on
\(\{\xi_i=0\}\), choose uniformly among the remaining eligible sides.
If \(P_i=1\), adjoin an auxiliary Bernoulli variable with the same
parameter, independently of the actual partner choice.  In this
construction the sequence \((\xi_i)\) is independent, because its
conditional success probabilities \(h_i\) are deterministic.  A
success is an actual strong closure whenever \(P_i\geq3\).  The
possible discrepancies occur only at times for which \(P_i\leq2\).

Let
\[
        \Theta:=\inf\{i:\xi_i=1\}.
\]
For every fixed \(t<b\),
\[
 \mathbb P\{\Theta\geq tN\}
   =\prod_{0\leq i<tN}\left(1-\frac1{N-2i-1}\right)
   =\sqrt{1-2t}+o_N(1).
 \tag{29}\label{eq:direct-survival}
\]
The last equality follows by taking logarithms: the sum of the
squared hazards is \(O(N^{-1})\), while the first-order sum is a
Riemann sum for
\(\int_0^t(1-2s)^{-1}\,ds\).
Consequently
\[
 \mathbb P\{a_1N\leq\Theta<a_2N\}=q+o_N(1).
 \tag{30}\label{eq:direct-window}
\]

Let
\[
        A_b:=\sum_{0\leq i<bN}\1_{\{P_i\leq2\}},
        \qquad
        \eta:=C_b\sqrt{\delta},
\]
where \(C_b\) is increased if necessary so that
\eqref{eq:direct-short-time} gives
\[
        \mathbb P\{A_b>\eta N\}
        \leq C_b\sqrt{\delta}+o_N(1).
\]
\revision{To obtain \eqref{eq:direct-bad-coin}, consider only the first
\(\lfloor\eta N\rfloor\) indices \(i<bN\) for which \(P_i\leq2\).
Whether exposure \(i\) is counted is known before its auxiliary variable is
sampled. Before time \(bN\), each auxiliary Bernoulli variable has success
probability at most \(C_b/N\), so the conditional union bound over these
counted times is at most
\(C_b\eta\). On \(\{A_b\leq\eta N\}\), they include every active short time.
Adding the preceding exceptional probability gives}
\[
 \mathbb P\left\{
   \xi_i=1\text{ for some \(i<bN\) with \(P_i\leq2\)}
 \right\}
 \leq C_b\sqrt{\delta}+o_N(1).
 \tag{31}\label{eq:direct-bad-coin}
\]
\revision{This use of the truncated count is valid although the set
\(\{i:P_i\leq2\}\) is random, since its indicator is known before the
variable \(\xi_i\) is sampled.}

\revision{On the event in \eqref{eq:direct-window}, outside the event in
\eqref{eq:direct-bad-coin}, the first auxiliary success gives a strong
closure unless \(\tau_{\mathrm{loop}}<\Theta\).}
Lemma~\ref{lem:direct-swallow}, with \(c=b\), bounds the latter
possibility by \(C_b\delta+o_N(1)\).  Thus
Lemma~\ref{lem:direct-orbit}, \eqref{eq:direct-window}, and
\eqref{eq:direct-bad-coin} give
\[
 \mathbb P\left\{
   |C_{\tau\chi}(j)|\geq \lceil a_1N\rceil
 \right\}
 \geq q-C_b\sqrt{\delta}-C_b\delta-o_N(1),
\]
because a closure at the integer time
\(\Theta\geq a_1N\) exposes at least
\(\lceil a_1N\rceil\) distinct orbit points.  This is
\eqref{eq:direct-clock-bound}.
\end{proof}

\begin{proof}[Proof of Proposition~\ref{prop:small}]
Choose any fixed \(0<a_1<a_2<b<1/2\), and let
\[
        q=\sqrt{1-2a_1}-\sqrt{1-2a_2}.
\]
Next choose \(\delta>0\) so small that all the preceding estimates
apply and
\[
        C_b\sqrt{\delta}+C_b\delta\leq q/2.
\]
Set
\[
        \delta_0:=2\delta,\qquad a:=a_1,\qquad p:=q/2.
\]
If at most \(\delta_0K=\delta N\) sides of \(H\) lie in \(1\)- or
\(2\)-cycles of \(\chi\), Lemma~\ref{lem:direct-clock} gives, uniformly
for every prescribed \(j\) outside those cycles,
\[
 \mathbb P\left\{
   |C_{\tau\chi}(j)|\geq \lceil a_1N\rceil
 \right\}
 \geq p-o_K(1).
\]
Since \(N=2K\), the event on the left implies
\(|C_{\tau\chi}(j)|\geq aK\).  This proves the proposition.
\end{proof}

\section*{\texorpdfstring{\revision{Declaration on the use of generative AI}}{AI declaration}}
The connection of the stirring model to random surfaces and, in particular, the work of Budzinski, Curien, and Petri that underpins the present article was brought to our attention through OpenAI ChatGPT 5.2. We also employed Anthropic Claude Sonnet 5 to generate the TikZ code for the figures and for copy editing. The authors take full responsibility for the contents of the manuscript.

\printbibliography

\end{document}